\documentclass[11pt]{amsart}
\usepackage{amsfonts,amsmath,mathrsfs,amssymb,amsthm,amscd,latexsym}
\usepackage[all]{xy}
\usepackage{enumerate}
\xyoption{all}
\usepackage{color}
\usepackage[active]{srcltx} %per fer cerca inversa
\newtheorem{thm}{Theorem}[section]
\newtheorem{lemma}[thm]{Lemma}
\newtheorem{lem}[thm]{Lemma}

\newtheorem{cor}[thm]{Corollary}
\newtheorem{prop}[thm]{Proposition}

\theoremstyle{definition}
\newtheorem{rem}[thm]{Remark}

\newtheorem{defn}[thm]{Definition}
\newtheorem{definition}[thm]{Definition}

\newcommand{\nc}{\newcommand}
\nc{\cD}{{\mathcal D}}
\nc{\cC}{{\mathcal C}}
\nc{\cE}{{\mathcal E}}
\nc{\cF}{{\mathcal F}}
\nc{\cG}{{\mathcal G}}
\nc{\cH}{{\mathcal H}}
\nc{\cI}{{\mathcal I}}
\nc{\cL}{{\mathcal L}}
\nc{\cP}{{\mathcal P}}
\nc{\cO}{{\mathcal O}}
\nc{\cV}{{\mathcal V}}
\nc{\cJ}{{\mathcal J}}
\nc{\cA}{{\mathcal A}}
\nc{\cZ}{{\mathcal Z}}
\nc{\cT}{{\mathcal T}}
\nc{\cX}{{\mathcal X}}
\nc{\cU}{{\mathcal U}}
\nc{\cR}{{\mathcal R}}
\nc{\bH}{{\mathbb H}}
\nc{\bA}{{\mathbb A}}
\nc{\bG}{{\mathbb G}}
\nc{\bC}{{\mathbb C}}
\nc{\bO}{{\mathbb O}}
\nc{\bI}{{\mathbb I}}
\nc{\bB}{{\mathbb B}}
\nc{\bY}{{\mathbb Y}}
\nc{\bK}{{\mathbb K}} 
\nc{\bX}{{\mathbb X}}
\nc{\bS}{{\mathbb S}}
\nc{\bE}{{\mathbb E}}
\nc{\bF}{{\mathbb F}}
\nc{\bZ}{{\mathbb Z}}
\nc{\bQ}{{\mathbb Q}}
\nc{\bN}{{\mathbb N}}
\nc{\bP}{{\mathbb P}}
\nc{\bL}{{\mathbb L}}
\nc{\bM}{{\mathbb M}}
\nc{\bT}{{\mathbb T}}
\nc{\bW}{{\mathbb W}}
\nc{\bU}{{\mathbb U}}
\nc{\bD}{{\mathbb D}}
\nc{\bJ}{{\mathbb J}}
\nc{\bV}{{\mathbb V}}
\nc{\bR}{{\mathbb R}}
\nc{\tC}{{\widetilde C}}
\nc{\lra}{{\longrightarrow}}
\nc{\la}{{\longrightarrow}}
\nc{\fr}{{\rightarrow}}
\nc{\co}{{\nabla}}
\nc{\cu}{{\overline{\nabla}}}

\title{The Fano normal function}
\author{Alberto Collino}
\address{
Dipartimento di Matematica \\
Universita di Torino \\
Via Carlo Alberto 10, 10123 \\
Torino, Italy}
\email{alberto.collino@unito.it}

\author{Juan Carlos Naranjo}
\address{
Departament d'\`Algebra i Geometria \\
Facultat de Matem\`atiques \\
Universitat de Barcelona \\
Gran Via 585 \\
08007 Barcelona, Spain }
\email{jcnaranjo@ub.edu}

\author{Gian Pietro Pirola}
\address{
Dipartimento di Matematica ``F. Casorati''\\
Universit\`a di Pavia\\
Via Ferrata 1\\
27100 Pavia, Italy}
\email{gianpietro.pirola@unipv.it}

\thanks{J.C. Naranjo has been partially supported by the Proyecto de Investigaci\'on MTM2009-14163-C02-01.G.P. Pirola
has been partially supported by Gnsasa and by MIUR PRIN 2009: \textup{Moduli, strutture geometriche e loro applicazioni};}

\begin{document}

\begin{abstract}
The  Fano surface  $F$ of  lines  in  the cubic threefold $V$ is naturally embedded in 
the intermediate Jacobian  $J(V)$, we call  ``Fano cycle'' the difference $F-F^-$, this is homologous to $0$
   in $J(V)$.   We study  the normal function on the  moduli space which 
computes the Abel-Jacobi image  of the Fano cycle.  By means of  
the related infinitesimal invariant  we can prove that 
the primitive part of the normal function is not of torsion. As a consequence
we get that, for a general $V$,
$F-F^-$ is not algebraically equivalent to zero in $J(V)$ (proved also by van der Geer and Kouvidakis \cite{vdGK} with different methods) and, moreover, that there is no divisor 
in $JV$ containing both $F$ and $F^-$ and such that these surfaces are homologically equivalent in the divisor.

Our study of the infinitesimal variation of Hodge structure for $V$ 
produces intrinsically a threefold $\Xi (V)$ in the Grasmannian of lines $\mathbb G$ in $\mathbb P^4.$  
We show that the infinitesimal invariant at $V$ attached to the normal function  gives a section of a natural 
 bundle   on $\Xi(V)$ and more specifically that this section 
vanishes exactly on  $\Xi\cap F,$ which turns out to be the curve in $F$ parameterizing  the ``double lines'' in the threefold.
We prove that this curve  
reconstructs  $V$ and hence we get a Torelli-like result: 
the infinitesimal invariant for the Fano cycle determines $V$.
\vskip 3mm
\noindent {\it R\'esum\'e.}
La surface de Fano $F$ des droites d'un threefold cubique $V$ est naturellement plong\'ee dans la vari\'et\'e Jacobienne interm\'ediaire $J(V)$, on va appeller ``cycle de Fano'' la diff\'erence $F-F^-$, ce cycle est homologue \`a $0$ dans $J(V).$ On \'etudie l'application normale, d\'efinie sur l'\'espace des modules, qui compute l'image d'Abel-Jacobi du cycle de Fano. Au moyen du correspondant invariant infinit\'esimal on prouve que la partie primitive de l'application normale n'est pas de torsion. Par cons\'equent on a que, si $V$ est g\'en\'erale, $F-F^-$ n'est pas alg\'ebriquement \'equivalent \`a $0$ dans $J(V)$ (ce qui a \'et\'e \'egalement prouv\'e par van der Geer et Kouvidakis  \cite{vdGK} avec une m\'ethode diff\'erente); de plus, dans $J(V)$ il n'y a aucun diviseur qui contienne $F$ ainsi que $F^-$ et tel que que les deux surfaces soient homologiquement \'equivalentes dans le diviseur m\^eme.

Notre \'etude de la variation infinit\'esimale de la structure de Hodge pour $V$ produit, de fa\c con intrins\`eque, un threefold $\Xi (V)$ dans la vari\'et\'e de Grassmann $G$ des droites dans $\mathbb P^4$. On fait voir que l'invariant infinit\'esimal en $V$ attach\'e \`a l'application normale donne une section du fibr\'e naturel sur $\Xi (V)$; plus pr\'ecis\'ement, cette section s'annule exactement sur $\Xi (V)\cap F$, ce qui on d\'ecouvre \^etre la courbe de $F$ qui param\`etre les ``doubles droites'' dans le threefold. On prouve que cette courbe reconstruit $V$ et on obtient, en conclusion, un th\'eor\`eme de type Torelli: l'invariant infinit\'esimal du cycle de Fano d\'etermine $V$. 


\end{abstract}

\maketitle
\pagestyle{plain}

%%%%%%%%%%%%%%%%%%%%%%%%%%%%%%%%%%%%%%%%%%%%%%%%%%%%%%%%%%%%%%%%%%%%%%%%%%%%%%%%%%%%%%%%%%%%%%%%%%%%%%%%%%%%%%%%%%%%%%%%%%%%%%%%%%%%%%%%%%%%%%%%%%%%%%%%%%%%%%%%%%%%%%%%%%%%%%%%%%%%%%%%%%%%%%%%%%%%%%%%%%%%%%%%%%%%%%%%%%%%%%%%

\section{Introduction }

We  present some considerations on the classical theme of cubic threefolds \cite{CG}  and specifically  on geometric properties encoded in 
the variations of their Hodge structure, \cite{Carlson etc.} and \cite{Voisin3}.  

The  Fano surface  $F$ of  lines  in  the cubic 
threefold $V$ is naturally embedded in the intermediate Jacobian  $JV$.  We define the
 ``Fano cycle''  to  be  the difference $F-F^-$, which   is homologous to $0$   in $JV$. 
Our  main interest  is the study of the normal function $\nu_F$ on the  moduli space of cubic threefolds  associated  with
the Abel-Jacobi class  of the Fano cycle, we call it the Fano normal function.  By means of a  somewhat detailed  analysis
 of  the related infinitesimal invariant  $\delta_{\nu}(F)$ we can prove that the primitive part of  $\nu_F$ is not of torsion.
 Our result  is stronger,  it  yields  further  the fact  that for  $V$ generic there is no map $\mu: W \to JV $ from a smooth variety $W$   of   dimension $4$, for which there is a  cycle $Z$ homologous to $0$ in $W$    
with  $\mu _{\ast} Z$ being  Abel-Jacobi equivalent to the Fano cycle.  We have the consequence that in this case    $F-F^-$ is not algebraically equivalent to zero on $JV$.

More precisely, given a family of Fano surfaces $\mathcal F\to B$ parameterized by an open set $B$ of the moduli space of cubic threefolds with a section $s$, we can define the cycle $\mathcal F-\mathcal F^-$ in the corresponding family of intermediate Jacobians of the cubic threefolds $\mathcal J\to B$. This provides a normal function $\nu_F$ which, via the decomposition in primitive cohomology, can be written as  $\nu_F=\nu^1+\nu^3+\nu^5$ (see \S 5 for the details). The pieces $\nu^3$, $\nu^5$ are independent of the choice of the section $s$. We prove (see Proposition \ref{prop:vanishing_mu_3}, Theorem \ref{collez}, and Theorem \ref{thm:alg_equiv}):
\begin{thm} The following holds:
\begin{enumerate}
 \item [a)] The primitive normal function $\nu^5$ is not of  torsion, moreover $\nu^3=0$.
 \item [b)] For a general cubic threefold, $F$ is not algebraically equivalent to $F^-$. 
 \item [c)] For a general cubic threefold $V$ there is no divisor  $W$ on  $JV$ which  contains  $F$ and $F^-$ and is such that $F-F^-\sim_{hom}0$ in $W$.
\end{enumerate}
\end{thm}

Part b) is not new. By studing the behaviour of the Fourier-Mukai transform in a rank one degeneration of Abelian varieties, van der Geer and Kouvidakis (see \cite{vdGK}, section 8) proved this statement. 
 
Following  Griffiths'  lead, see (6.32)  of  \cite{GriffIII},  we show how to reconstruct $V$ from $\delta_{\nu}(F)$. We came to our proof motivated by the ideas  of  \cite{CP}; there it was  proved,   for curves of genus $3$,  that   the infinitesimal invariant $\delta_{\nu}(C)$ gives back  the curve $C$ (recall that  the normal function $\nu_C $  computes  the  primitive  Abel-Jacobi class  of the cycle  $C-C^-$ in $J(C)$). In the spirit of \cite{CG} we are able to relate the infinitesimal behaviour of the Fano normal function with the projective geometry of the cubic threefold in $\mathbb P^4$.
To be concrete, our study of the infinitesimal variations of Hodge structure for $V$ produces intrinsically a 
threefold $\Sigma (V)$ in  the Grassmannian of lines in $\mathbb P^4.$  
These lines  appear already in   \cite{CG} under the name of   
{\it { lines  of second type}}.
We show that $\delta_{\nu} (V)$    
yields  a section of a natural   bundle   on $\Sigma (V)$ and more specifically that this section 
vanishes exactly on  $\Sigma (V)\cap F,$ which is the  curve parameterizing the double lines
in $V$. We notice that the curve of double lines 
reconstructs  $V$ and rephrase this last  part of  our work  as a Torelli-like result (see Theorem \ref{torelli}): 
\begin{thm}
The infinitesimal invariant for the Fano cycle determines $V$. 
\end{thm}

Normal functions have been extensively studied,  of special relevance  to us here  is Nori's work,  \cite[Section 7.5]{Nori}. He 
proved that for $g\geq 4$ there are no interesting normal functions over  $\mathcal A_g  $, the moduli space
of Abelian varieties of genus $g$.
On the other hand N.  Fakhruddin  \cite{Fakhr}, using Nori's connectivity theorem,
showed that nevertheless the Griffiths groups of codimension $3$ and $4$ are of infinite rank on the generic Abelian variety 
of dimension $5$, although they are in the kernel of the Abel-Jacobi mapping.
By   the  Torelli theorem in  \cite{CG},  the moduli space of cubic threefolds embeds in  $\mathcal A_5$. Keeping in mind this framework then our result about the normal function might be interpreted  as the statement
that intermediate jacobians of cubic threefolds lie  inside  a certain Noether-Lefschetz-like sub-locus of $\mathcal A_5$.  

Turning instead to the case of Jacobians of curves,  consider the symmetric product of a curve $C$, and specifically its image,   the surface $W_2 \subset J(C)$.  The normal function associated with the  {\it{  primitive }} Abel-Jacobi class  of $W_2 -W_2^{-}$   is trivial.  This statement, which is also proved in the remark (\ref{hain}) below, follows from the general results   of R. Hain \cite{Hain}, who  proved that over $\mathcal M_g$  all interesting normal function come from the  normal function $\nu_C $  of the   cycle  $C-C^-$ in $J(C)$. Now Allcock-Carlson-Toledo (see \cite{Allcock}) defined a period map for cubic threefolds
which takes values in a ball quotient of dimension 10. A theorem
of C. Voisin \cite{Voisin1} implies that this is an open embedding. We think that our results   provide some  support to the possibility  that   the structure of the group
of normal functions over the moduli space of cubic threefolds should be a cyclic group, as it is the case for the Jacobian  situation. 
We point out that our work shows that   $\nu_F$ is certainly  a truly distinguished element of this group.

The structure of the paper is as follows: in section \S 3 we give some  computations in the Jacobian ring associated with the equation of a cubic threefold.
In particular we study the infinitesimal deformations of rank $2$ and we characterize them in terms of lines in $\mathbb P^4$. 
Next in section \S 4 we use the detailed analysis of the differential forms on $F$ given in  \cite{CG} and \cite{T} to compute the adjoint class form introduced in \cite{CP} and \cite{PiZu}.
This is the basic tool we employ in section \S 5 to compute the Griffiths' infinitesimal invariant of the Fano normal function.
The information we get allows us to prove the results summarized in Theorem 1.1 above.  In section \S 6 we do a careful analysis of the geometry of the curve of double lines needed to prove our Torelli-like Theorem 1.2. 

%%%%%%%%%%%%%%%%%%%%%%%%%%%%%%%%%%%%%%%%%%%%%%%%%%%%%%%%%%%%%%%%%%%%%%%%%%%%%%%%%%%%%%%%%%%%%%%%%%%%%%%%%%%%%%%%%%%%%%%%%%%%%%%%%%%%%%%%%%%%%%%%%%%%%%%%%%%%%%%%%%%%%%%%

\vskip 5mm
{\it Acknowledgments.} We would like to thank the referee for many helpful suggestions and careful reading of the manuscript.

\section{Notation}

 \begin{itemize}
  \item We work over  $\mathbb C$, the field of the complex numbers,
  \item $V$ is a smooth cubic threefold in $\mathbb P^4$ of equation $E \in S:=\mathbb C[z_0,\dots,z_4]$,
  \item $JV$ the intermediate Jacobian of $V$ endowed with the natural principal polarization 
$\Theta :=\Theta_{JV}$ 
        (see \cite{CG}, \S 3),
  \item $\bG:=Grass(2,5)$ is the Grassmann variety of lines in $\mathbb P^4\,.$ The standard  Pl\"ucker coordinates $p_{i,j}$  give the embedding  $\bG  \subset \mathbb P^9 \, ,$ 
  \item The hyperplane $p_{i,j} = 0 $ cuts on $G$ the Schubert divisor $D_{\pi(i,j)}$, which is the locus of elements  $r  \in \bG$ such that the corresponding line $l_r\subset \bP^4$ intersects the fixed  plane $\pi(i,j)$ of equation $z_i = 0 = z_j$,
  \item $F\subset \bG$  is the Fano surface of lines contained in $V$,
  \item A line $r\in F$ is called {\it double} if there exists a $2$-plane intersecting $V$ in $2l_r+l_s$. 
Double lines define a curve $\cD$ in $F$. 
 \end{itemize}
 %%%%%%%%%%%%%%%%%%%%%%%%%%%%%%%%%%%%%%%%%%%%%%%%%%%%%%%%%%%%%%%%%%%%%%%%%%%%%%%%%%%%%%%%%%%%%%%%%%%%%%%%%%%%%%%%%%%%%%%%%%%%%%%%%%%%%%%%%%%%%%%%%%%%%%%%%%%%%%%%%%%%%%%%%%%%%%%%%%%%%%%%%%%%%%%%%%%%%%%%%%%%%%%%%%%%%%%%%%%%%%%%
 \section{Infinitesimal variations of Hodge structures  for   cubic threefolds}
We are concerned here  with  infinitesimal variations of Hodge structures for  the cubic threefold $V$ and for related objects,
namely the  intermediate Jacobian $JV$ and the Fano surface $F$ of lines on $V$.
Everything turns around the linear map $ \delta :T \to Hom ( H^{2,1}(V) , H^{1,2}(V))$ where $T$ 
is the tangent space at $[V]$ to the moduli space of cubic threefolds. This is understood and controlled by means of  
polynomial computations because of  Griffiths'  results on  the    Hodge filtration for hypersurfaces, \cite{PRI} and \cite{Voisin3}, \S 6.

We briefly recall the  basic facts we need for the  smooth cubic threefold   $V$ in  $\mathbb P^4$.  
Denote by $J\subset \mathbb C[z_0,\dots,z_4]=S=\oplus _{i\ge 0} S^i $ the Jacobian 
ideal generated by the partial derivatives $\frac {\partial E}{\partial z_i}\in S^2$.
The Jacobian ring is the graded ring $R:=S/J=\oplus_iR^i$.

One has $\dim R^0=\dim R^5=1$, $\dim R^1=\dim R^4=5$, $\dim R^2=\dim R^3 =10$  and the others are zero. 
Moreover the product of polynomials induces perfect pairings
\begin{equation*}
 R^i \otimes R^{5-i}\lra R^5=\mathbb C, \quad i=0,\dots,5.
\end{equation*}
Observe that $R^1=S^1$ is the vector space of equations of hyperplanes on $\mathbb P^4$, in other
words $\mathbb P^4=\mathbb P (R^{1\,*})\cong \mathbb P(R^4)$.
We  remark  that  the tangent space $T$ to the moduli   of cubic threefolds is identified as follows:
 \begin{equation*}
T:=  H^1(V,T_V)\cong R^3 .
 \end{equation*}

The  canonical isomorphisms of $R^i$ with the pieces of the Hodge decomposition of $V$ are:

$$H^{a,3-a}(V)\cong R^{7-3a}, \quad \forall a. $$ 
namely:
 $$ H^{3,0}(V)\cong R^{-2}=0 \,,\quad H^{2,1}(V)\cong R^1  \,,\quad H^{1,2}(V)\cong R^4  \,,\quad H^{0,3}(V)\cong R^7=0.$$ 
 
All the cup-product maps correspond via these isomorphisms to products of classes of polynomials.  For example 
the cup product map $H^1(V,\Omega^2_V)\otimes H^1(V,T_V)\stackrel{\cup}{ \lra } H^2(V,\Omega ^1_V)$
can be seen as the multiplication $R^1\otimes R^3 \lra R^4$.
 
The beautiful results of  \cite{CG} show how the Hodge decomposition of $V$ is linked  with that of $F$
via the correspondence given by the universal family of lines $\mathcal{L}\subset F\times V$. In particular, there are isomorphisms  
$H^{2,1}(V)\cong H^{1,0}(F)$ and  $H^{1,2}(V)\cong H^{0,1}(F)$. Even more, there is an isomorphism of 
principally polarized
 Abelian varieties $JV\cong Alb(F)$. Clemens-Griffiths  and Tjurin (cf. \cite{T})  completed  Fano's work. In particular, concerning 
the canonical system of $F$,   they    gave modern proofs of  the isomorphisms 
$H^0(F,\Omega^2_F)\cong \Lambda ^2 H^0(F,\Omega^1_F) \cong \Lambda ^2 R^1$.
Furthermore    the  space of  first order   deformations   of  $V$  is  the same as    that of   $F$, i.e. 
$H^1(V,T_V)\cong H^1(F,T_F)\cong R^3$.

In order to apply the tool of adjoint forms in a later section we need to have at our disposal   elements in $R^1=H^0(F,\Omega ^1_F)$ which are orthogonal to a first order deformation, i.e. to some  $\xi \in R^3 $.  
Associated with $ \delta :T \to Hom ( H^{2,1}(V) , H^{1,2}(V))$, i.e.
$ \delta :R^3 \to Hom (R^1, R^4)$, there are determinantal varieties
\begin{equation}\label{def_Xi}
  \Xi _j :=\{ \xi \in  \mathbb PR^3 \,\vert \,  rank ( \delta(\xi)) \leq j \} \subset  \mathbb P R^3.
\end{equation}
  
Now we have   $  \Xi _0  =    \Xi _1  =\emptyset  ,$ indeed: 
\begin{lem} \label{rank_xi}
If   $ 0 \neq\xi  \in R^3$, then   $rank (\delta (\xi ))  \geq 2.$
\begin{proof}
 Let  $K_j(\xi):=\ker (R^j \stackrel{\cdot \xi}{\lra} R^{j+3})$. We want to see that
the dimension of $K_1(\xi)$ is $\le 3$.
 We have  $\dim K_2(\xi)=9$, since 
$R^2\otimes R^3\longrightarrow R^5$ is a perfect pairing and $\xi \ne 0$. 
Observe that  $ \dim K_1(\xi ) = 5  $ is impossible, as $\xi \neq 0 $. 
Recall next that $S^2$, the vector space of the equations of the quadrics in $\mathbb P^4$,
 has dimension $15$. If   $ \dim K_1(\xi) = 4 $, then     $ K_1(\xi) \cdot S^1\subset S^2 $ 
is the family of the quadrics vanishing on a point, 
the  base locus of $K_1(\xi)$, and therefore   $\dim K_1(\xi) \cdot S^1 = 14$.  
Since $ K_1(\xi) \cdot S^1 $ 
is contained in the hyperplane in $S^2$  pull back of  $K_2(\xi )\subset R^2$, 
then we see that  $K_1(\xi) \cdot S^1 $ must contain $J_2$.  
In this case the base point should be singular for $V$,  because the partial derivatives vanish on it. 
This is  a contradiction since $V$ is  smooth.
\end{proof}
\end{lem}

The next natural goal is to determine  $\Xi _2$, i.e. to understand 
 for which $\xi \in R^3$ the dimension of $K_1(\xi)$ is exactly $3$. 
Our answer  is in terms of the associated line, cut by the corresponding system 
of hyperplanes in $\mathbb P^4$.
Consider the grassmannian $\bG':=Grass(3, R^1 )$ 
of $3$-dimensional subspaces of $R^1$ which is isomorphic 
to the Grasmannian of lines $\bG =Grass(2,R^4)$:
\begin{equation*}
 \begin {aligned}
  &\bG \,\lra \,\bG' \\
  &r\mapsto W(r)=\{\text{hyperplanes containing the line }l_r\}. 
 \end {aligned}
\end{equation*}
We define
\begin{equation*}
 \Gamma=\{(r,[\xi])\in \bG \times \mathbb P R^3 \,\vert \, \eta \cdot  \xi=0 \quad \forall \eta \in W(r)\}
=\{(r,[\xi])\in \bG \times \mathbb P R^3 \,\vert \,K_1(\xi)=W(r)\}.
\end{equation*}
We write  $\pi_1, \pi_2$ the projections of $\Gamma $ into the factors, thus 
$\Xi _2 =\pi_2(\Gamma)$  and we set $\Sigma:=\pi_1(\Gamma)\subset \bG$.

Let $r\in \bG$ be an element of the Grassmannian and let $l_r$ be
the corresponding line in $\mathbb P^4$. 
Using Macaulay duality and recalling that $W(r)\cdot S^1 \subset S^2$ 
is the space of quadrics containing $l_r$ (hence of dimension $12$),
one can easily check the following equivalences: 
\begin{equation*}
\begin {aligned}
 & r \in \Sigma \iff W(r)\otimes R^1\to R^2 \text{ is not  surjective }\\
& \text {(because the image is contained in the hyperplane } K_2(\xi)) \\
&\iff W(r)\cdot S^1+J_2 \subsetneq  S^2 \iff  \dim (W(r)\cdot S^1) \cap J_2  \geq 3,
\end {aligned}
\end{equation*}
which is the case if and only if  the restriction of  $ J_2 $ to $l_r$ has dimension at most $2$.
Notice that  $ \dim (W(r)\cdot S^1) \cap J_2  \geq 4 ,$ is impossible.
Indeed in this case the restriction of $J_2$ to $l_r$  would be of dimension $1$,
and therefore the zero locus of the ideal $J=(\partial E/\partial z_i)$ would not be empty,
hence $V$ would be singular. We have then proved the following:
\begin{lem}\label{Caract_Sigma}
Given a line $r\in \bG$ the following are equivalent:
\begin{enumerate}
 \item [a)] $r \in \Sigma $.
 \item [b)] $ W(r)\cdot R^1$ is a hyperplane in $R^2$.
 \item [c)] $\dim (W(r)\cdot S^1) \cap J_2  = 3$.
 \item [d)] $\dim (J_2 |_{l{_r} }  )= 2.$
\end{enumerate}
\end{lem}

\begin{rem}\label{dual}
 The lines we have found are classical objects of study, called {\it special} by   Fano and {\it  lines of second type  } by  \cite{CG}.
Since we need to appeal to  this  property   later, we recall here  that a line $l\subset \mathbb P^4$ is of second type exactly 
if the dual  map 
\begin{equation*}
\begin{aligned}
d:\mathbb P^4  \,\,&\lra \quad \mathbb P^{4\,*} \\
 p & \mapsto \left(\frac{\partial E}{\partial z_0}(p): \dots  : \frac{\partial E}{\partial z_4}(p)\right)
\end{aligned}
\end{equation*}
 maps $l$  to a line $l'$ and $l\lra l'$ has degree $2$. 
The base locus of the corresponding dual pencil is then a plane.
\end{rem}
 
Next Lemma deals with the geometry of $\Sigma$:
\begin{lem} One has:
  \begin{itemize}
   \item [a)] The projection $\pi_1:\Gamma \to \Sigma$ is bijective. 
   \item [b)] $\dim \Gamma=\dim \Sigma=3.$
  \end{itemize}
  \end{lem}
  \begin{proof}
To prove a) we appeal to  Lemma (\ref{Caract_Sigma}): $r\in \Sigma$ if and only if  $W(r)\cdot R^1$ is a hyperplane
in $R^2$. Then there is a unique element, up to constant, $\xi \in (R^2)^*\cong R^3$ corresponding to 
this hyperplane and
by construction $W(r)=K_1 (\xi)$. 

 To show b) we consider the map $\rho: R^3\to Sym^2 R^4$ induced by the product $R^1\otimes R^3 \to R^4$
and the isomorphism $R^1\cong (R^4)^{*}$. By definition
\begin{equation*}
[\xi ]\in \Xi_2 \iff  rk(\rho (\xi)) \leq 2 . 
\end{equation*}
In the space of quadrics in $\mathbb P^4$, i.e. in  $\bP(Sym^2 R^4)=\bP^{14}$,  the variety of the 
quadrics of rank $2$   is of codimension $6$, being the space of couples of hyperplanes.  Since  
  $\rho$ is injective then we have   $\dim \Xi_2 \geq 3$. 
Recalling that  in the Segre product  $\mathbb  P^{4 *} \times \mathbb P^{4*}$ every effective cycle 
 of dimension $4$ intersects the diagonal,
we see that  if $\dim \Xi_2\geq 4$, 
then the variety  $\rho ( \Xi_2)\subset \bP (Sym^2 R^4)$  would  intersect the image of the  diagonal which corresponds 
to  the locus of rank one quadrics. 
This contradicts the fact that $rank(\rho(\xi))>1$ (Lemma (\ref{rank_xi})). 
We finally notice that  $\pi_2 : \Gamma \to \Xi $ is an isomorphism,  therefore the result follows.
\end{proof}
   
 The lines from  $\Sigma$ which are contained in $V$ can be geometrically characterized by their tangency property, see  \cite{CG} (6.7) : 
\begin{lem} A line 
 $r\in F\cap \Sigma$ if and only if there is  a plane tangent to $V$ along $l_r$, which then we call a double line for $V$.
\end{lem}

  Now we discuss some particular planes contained in  $\Sigma$. Consider  an Eckardt point $p \in V$. 
This means that through $p$ there is a curve of lines on $V$, 
and more precisely the tangent hyperplane $T_p(V)$ cuts $V$ in a cubic cone, with smooth base $E_p.$ These points do not appear on a general cubic threefold.
If we take $z_4 = 0 $ to be the equation of $T_p V$, and $p=(1:0:\dots:0)$, 
then  the equation for $V$ is of type $ K(z_1,z_2,z_3) + z_4 Q( z_0,z_1,z_2,z_3,z_4) = 0 .$
Given  a  line $l$  such that $ p \in  l  \subset  T_p V $ one can assume  that $l$ has equations $z_2= z_3 = z_4 = 0 $.  In this way it is easy to see that  
$l\in \Sigma$, and therefore writing  $ \pi _ p $ for  the plane of such lines  
we have $ \pi _p \subset \Sigma \subset   \bG$ and the intersection $F \cap  \pi _p = E_p \, .$

%%%%%%%%%%%%%%%%%%%%%%%%%%%%%%%%%%%%%%%%%%%%%%%%%%%%%%%%%%%%%%%%%%%%%%%%%%%%%%%%%%%%%%%%%%%%%%%%%%%%%%%%%%%%%%%%%%%%%%%%%%%%%%%%%%%%%%%%%%%%%%%%%%%%%%%%%%%%%%%%%%%%%%%%

\subsection{Geometry of the differential forms on $F$}  In \cite{CG} we find the following important properties of the Fano surface $F\subset \bG\subset \mathbb P^9$ which are crucial tools for our work.
\begin{thm} (Clemens-Griffiths).
\begin{enumerate} 
\item [a)] There is a canonical isomorphic between the Albanese variety $Alb(F)$ and the intermediate Jacobian
$JV$ of $V.$
 \item [b)]  The Albanese map $a:F\to JV= Alb(F)$  coincides with the Abel-Jacobi on $V$ and it is an embedding; 
 \item [c)] (Tangent bundle theorem)  $T_F$ is the restriction to $F$ of the tautological rank 
$2$ bundle on the Grassmannian $\bG$. The Gauss map  $g$ 
attached to $a$ 
is the tautological embedding $F\hookrightarrow \bG$.  \item [d)] The canonical embedding $F\hookrightarrow \mathbb P^9$
is the composition of $g$ with the Pl\"ucker embedding  $\bG\hookrightarrow \bP^9. $
 \end{enumerate}
\end{thm}

 Consider $H\in R^1\setminus\{0\}$ and the corresponding  hyperplane in $\bP^4$, and let $\omega_H$ be the corresponding 
holomorphic 1-form on $F.$ The zero locus of 
 $\omega_H$ is the set of lines of $F$ contained in $H$. 
 When $H$ is not tangent
 to $V$ these are the $27$ lines on the cubic surface, this gives $c_2(F)=27.$

%\label{adjoint}

\begin{defn}\label{schubform} Given a plane $\pi=H_1\cap H_2 \subset \mathbb P^4 $ we write
$$\Omega_{\pi } =\omega_{H_1}\wedge \omega_{H_2}\in H^0(F,K_F),$$
which is determined  up to a constant.
\end{defn}

We denote by $Z(\pi )$ the canonical divisor on $F$ corresponding to  $\Omega_{\pi }$. 
Then  $Z(\pi)$  is the locus  of the lines which intersect $\pi$,
because it  is the restriction to $F$  of the  Schubert  divisor associated with $ \pi $.

Following \cite{CG} we define $C_t\subset F$ to be  the closure of the set of lines in 
$V$ which  intersect a different fixed line $l_t \subset V$.
Then $t\in C_t$ if and only if $t$ is a double line. 

\begin{defn}
Three lines $l_{t_1},l_{t_2},l_{t_3}$ of $V$ form a {\em triangle} if  they are cut on $V$ by a plane $P:$  
$$ V\cdot P= l_{t_1}+l_{t_2}+l_{t_3}.$$
A triangle is called {\it rare} if the lines of the triangle are incident at the same point.
\end{defn}

\begin{rem} Employing  classical terminology   \cite{Dol} we see that  the parabolic points on $V$ coincide with the vertices of  rare triangles. The locus of parabolic points 
is the intersection of $V$ with its Hessian hypersurface, which is of degree $5$,  and thus  a 
line on $V$  contains $5$ vertices of rare triangles, counted with multiplicities. For a parabolic point $P$ there are either $2$  or  $1$ or $\infty$  rare triangles with vertex $P$.
The  first case is when  $P$ is a biplanar point. The second if it is uniplanar, as a singular element of  the cubic surface cut on $V$ by its tangent hyperplane.  The last case goes under the name of Eckardt point.
 It happens when
the tangent section is a cubic cone.  \end{rem}
 
Using the definition  of a triangle we  can state from   \cite{Fano1}:

\begin{thm} (Fano).
The following hold:
$C_t^2=5$, $p_a(C_t)=11$ and 
 for any triangle $l_{t_1},l_{t_2},l_{t_3}$  the cycle 
$C_{t_1}+C_{t_2}+C_{t_3}  $ is a canonical divisor on $F$, so that  numerically $K_F=3C_t.$
\end{thm}

Rare  triangles are  useful  for   the proof of the following:
\begin{prop} 
Let $r\in R$. The curve $C_r$ is contained in $Z(\pi)$ if and only if $l_r\subset \pi.$
\end{prop}
\begin{proof} 
If $l_r\subset \pi$ it follows immediately that $C_r\subset Z(\pi).$
 Now assume that $C_r \subset Z(\pi)$, so that any line $l_t\neq l_r$ of $V$  that intersects $l_r$ intersects also $\pi.$
 First we show that $l_r\cap \pi \neq \emptyset$ 
which we prove  by contradiction.
We take any triangle $l_r,l_{r_1},l_{r_2}$  of which  $l_r$ takes part,   and we proceed to show  first that the vertex 
 $l_{r_1}\cap l_{r_2}$ is in $\pi.$ Otherwise consider  then 
$P=l_r\vee l_{r_1} \vee l_{r_2},$  the plane of the triangle. 
By assumption   $l_{r_1}\cap \pi \neq \emptyset   \neq l_{r_2}\cap \pi $. If the two points of intersection
are different then they span a line in  $\pi \cap P$, which  intersects $l_r$, and thus $ l_r\cap \pi \neq \emptyset\, .$ So $\pi$
contains all the vertices of the triangles along $l_r.$ Taking now  a rare triangle with vertex on $l_r,$
we see that this point  belongs to $\pi \cap l_r.$\\
Having proved that $l_r\cap \pi \neq \emptyset$ 
we show next that $l_r\subset \pi$. If not,  the linear span $H=\pi \vee l_r$ is a hyperplane which satisfies that
the general point  $s\in C_r$ has the property that $l_s\subset H$, and then this is true in fact for all the points in $C_r$.
On the other hand it is obvious that there are lines on $V$ which 
intersects $l_r$ but do not lie in the fixed hyperplane $H$. 
\end{proof}
 
\begin{defn}\label{def_base_locus}
Given a line  $l_s$ in $\bP^4$ we consider  the locus of lines in $V$ which intersect it:
$$D_s:=\{r\in F \,\vert \,l_r\cap l_s\neq \emptyset\} \subset F. $$
 Clearly   \begin{equation*}
D_s=\{r\in F \, \vert \, \Omega_{\Pi}(r)=0 \text{ for all $2$- planes } \Pi\supset l_s\}. 
\end{equation*} \end{defn}

More precisely,   $D_s$ is the subscheme cut on $F$  by the net  (two dimensional linear system)  of Schubert  hyperplanes 
associated with the net of 2-planes which support $l_r$.

We have 
\begin{cor}\label{base_locus_r_in_F}
$D_s\supset C_r \iff r=s$.
\end{cor}
\begin{proof} Direct  consequence of the proposition.
\end{proof}

 Consider inside $F$ the set $W_p$ of lines on $V$ which pass through a point $p$. Out of
a finite number of points (which are actually the Eckardt points) this has a natural structure of a finite scheme of
length $6$. 
For an Eckardt point,  $W_p$  parameterizes  the lines in a cone of vertex $p$ and basis the plane smooth 
cubic  $E_p$. In this situation  we shall identify   $W_p$   with $E_p$.

\begin{prop} \label{finite_base_locus}
  For   $s\in \Sigma - F $, one has:
\begin{enumerate}
\item [a)] If $l_s$  does not contain any Eckardt point then the base locus  $D_s$ is  of dimension $0$.
\item [b)] If $l_s$  contains some  Eckardt points $ \{p_i \} $   then   $D_s$ 
contains the  cubic curves $ E_{p_i} $ and possibly  a residual  finite set $Z$ .
\end {enumerate}
Hence, for  any  element $s\in \Sigma - F $, the scheme $D_s$  is not of dimension  zero only when  $l_s$ meets  an Eckardt point.
\end{prop}
 
%%%%%%%%%%%%%%%%%%%%%%%%%%%%%%%%%%%%%%%%%%%%%%%%%%%%%%%%%%%%%%%%%%%%%%%%%%%%%%%%%%%%%%%%%%%%%%%%%%%%%%%%%%%%%%%%%%%%%%%%%%%%%%%%%%%%%%%%%%%%%%%%%%%%%%%%%%%%%%%%%%%%%%%%
\section{The method of adjoint forms on  the Fano surface}
The main ingredient  which   we  employ    to compute the infinitesimal invariant of  our normal function 
is the  tool of the {\em adjoint form} due to \cite{PiZu}. We begin  this section by recalling the relevant definitions and the basic  results.  Next  we derive a sharp result  on the properties of the  adjoint form   in the setting of the Fano surface, which  will turn out to be of convenient use in later developments.

%%%%%%%%%%%%%%%%%%%%%%%%%%%%%%%%%%%%%%%%%%%%%%%%%%%%%%%%%%%%%%%%%%%%%%%%%%%%%%%%%%%%%%%%%%%%%%%%%%%%%%%%%%%%%%%
%%%%%%%%%%%%%%%%%%%%%%%%%%%%%%%%%%%%%%%%%%%%%%%%%%%%%%%%%%%%%%%%%%%%%%%%%%%%%%%%%%%%%%%%%%%%%%%%%%%%%%%%%%%%%%%
%%%%%%%%%%%%%%%%%%%%%%%%%%%%%%%%%%%%%%%%%%%%%%%%%%%%%%%%%%%%%%%%%%%%%%%%%%%%%%%%%%%%%%%%%%%%%%%%%%%%%%%%%%%%%%%
\subsection{The adjoint form and the adjoint Theorem}
Let $X$ be a projective smooth surface and let
\begin{equation*}
 X_{\varepsilon}\lra Spec \, \mathbb C[\varepsilon]/(\varepsilon^2)
\end{equation*}
be an infinitesimal deformation of the variety $X$ with class $\xi \in H^1(X,T_X)$. 
The coboundary map $\partial_{\xi}: H^0(X,\Omega ^1_X)\lra H^1(X,\cO_X)$ attached to the exact sequence
\begin{equation}\label{exact_seq_defo}
0\lra \cO_X \lra \Omega ^1_{ X_{\varepsilon}\vert X}\lra \Omega^1_X \lra 0,
\end{equation} 
corresponds to cupping with $\xi$. Therefore a global form $\eta \in H^0(X,\Omega^1_X)$ lifts to
$\tilde {\eta }\in H^0(X, \Omega ^1_{ X_{\varepsilon}\vert X})$ if and only if
 $\partial _{\xi }(\eta)=\xi \cup \eta=0$. 

Assume that there exist $3$ global $1$-forms on $X$, $\eta_1,\eta_2,\eta_3$ such that 
$\xi \cup \eta_i=0$ for all $i$ and choose
$\tilde {\eta }_i$ some liftings to $X_{\varepsilon}$. Put $W=\langle \eta_1,\eta_2,\eta_3\rangle 
\subset H^0(X,\Omega^1_X)$.

Then  $\tilde {\eta}_1 \wedge \tilde {\eta}_2 \wedge \tilde {\eta}_3$ belongs to
 $ \Lambda ^3 H^0(X,  \Omega^1_{ X_{\varepsilon}\vert X})$. By
 taking determinants in the exact sequence (\ref{exact_seq_defo}), one gets that 
$\Lambda ^3 \Omega^1_{ X_{\varepsilon}\vert X}\cong \omega_X$ and therefore
we can define $\omega _{\eta_1,\eta_2,\eta_3,\xi}$ to be the image of 
$\tilde {\eta}_1 \wedge \tilde {\eta}_2 \wedge \tilde {\eta}_3$ by the natural map:
\begin{equation*}
\Lambda ^3H^0(X,\Omega^1_{X_{\varepsilon}}) \lra H^0(X,\Lambda ^3\Omega ^1_{X_{\varepsilon}}) \cong 
 H^0(X, \omega_X).
\end{equation*}
We say that this is an {\em adjoint form} attached to $\eta_i$ and $\xi$. 
It is not hard to see that this form is well-defined up to 
adding elements of $W^2:= \psi (\Lambda ^2W)\subset H^0(X,\omega_X)$, where
\begin{equation*}
 \psi:\Lambda^2 H^0(X,\Omega ^1_X) \lra H^0(X, \omega_X).
\end{equation*}
So we can consider the {\em adjoint class}
\begin{equation*}
 [\omega _{\eta_1,\eta_2,\eta_3,\xi}] \in H^0(X,\omega_X)/W^2.
\end{equation*}
Finally let $B$ be the fixed base  divisor of the linear system $\vert \,W^2\, \vert \subset 
\vert \,\omega_X \,\vert$ and let  $Z$ be  the scheme which is the base locus of the moving part of the
 system. We recall  the basic result on the adjoint class:
\begin{thm}\label{adjoint_thm}(\cite{PiZu}, 1.5.1. and 3.1.5).
 Assume that $W^2\ne 0$. Then:
\begin{enumerate}
 \item [a)] $[\omega _{\eta_1,\eta_2,\eta_3,\xi}]\in j (H^0(X,\omega_X (-B)\otimes_{ \cO_X}{\mathcal I}_{Z}))/W^2$
 where
$$ j: H^0(X,\omega_X (-B)\otimes_{ \cO_X}{\mathcal I}_{Z  } ) \to  H^0(X,\omega_X)$$
is the natural inclusion.
\item [b)] 
If $ [\omega _{\eta_1,\eta_2,\eta_3,\xi}]=0$, then $\xi$ belongs to the kernel of
\begin{equation*}
 H^1(X,T_X)\lra H^1(X,T_X(B)).
\end{equation*}
\end{enumerate}
\end{thm}
Part b) is a tool to check  non-vanishing. Indeed, if $\vert \,  W^2 \, \vert$ is a 
non empty linear system without base divisor, then $ [\omega _{\eta_1,\eta_2,\eta_3,\xi}]\ne 0$.  

It is possible and sometimes useful to choose  a special representative of the adjoint class; this can be done by the requirement that it is orthogonal to $W^2.$  
\begin{definition} The standard representative  of the class     $[\omega _{\eta_1,\eta_2,\eta_3,\xi}]$   is the only form in the given class which is orthogonal to $W^2$.  We   write it   simply    $\omega_{\eta_1,\eta_2,\eta_3,\xi},$  with the understanding   that  it satisfies the property:

\begin{equation} \label{stand}
\int_X\omega_{\eta_1,\eta_2,\eta_3,\xi}\wedge \overline{\beta_1\wedge \beta_2}=0
\end{equation}
for all  $\beta_1$ and $\beta_2$ are in $W$. Clearly $[\omega_{\eta_1,\eta_2,\eta_3,\xi}]=0\iff \omega_{\eta_1,\eta_2,\eta_3,\xi}=0$
\end{definition}

%%%%%%%%%%%%%%%%%%%%%%%%%%%%%%%%%%%%%%%%%%%%%%%%%%%%%%%%%%%%%%%%%%%%%%%%%%%%%%%%%%%%%%%%%%%%%%%%%%%%%%%%%%%%%%%%%%%%%%%%%%%%%%%%%%%%%%%%%%%%%%%%%%%%%%%%%%%%%%%%%%%%%%%%
\subsection{The adjoint class for the special lines.}  

We explain how the preceding machinery works in case of the Fano surface $F$ and for a chosen deformation $\xi $ 
associated with a line $r$  belonging to $\Sigma$.  
In our situation the Hodge pieces  $H^{a,b}(F)$ of interest  to us can be written in terms  of the Jacobian ring (section \S 3), 
in particular the space of canonical forms on $F$  is  identified with $\wedge ^ 2 R^1.$ 
Having chosen  $r\in \Sigma $   we consider next a basis $\eta_1,\eta_2,\eta_3 $ of $W(r)$. The  element  $\xi \in R^3$,
 unique up to constant, is an infinitesimal variation  with kernel  $K_1(\xi)=W(r)$. We set 
\begin{equation*}
  W(r)^2:=\langle \eta_1\wedge \eta_2,\eta_2\wedge \eta_3,\eta_3\wedge \eta_1 \rangle \subset \wedge ^ 2 R^1\, \, .
\end{equation*}
 Observe that the base locus of the linear system $\vert W(r)^2\vert$ is the sublocus $D_r$ defined in
(\ref{def_base_locus}).
The goal of this section is to understand for which lines  $r\in \Sigma $ the adjoint class
$[\omega _{\eta_1,\eta_2,\eta_3,\xi}]$ is zero. In this case we say that ``the class vanishes on the line''. The vanishing
does not depend on the   choice of the basis we have just made.

\begin{rem}\label{adjoint_as_section}  Consider the universal rank $3$ subbundle on  $Grass(3, R^1 ) =: \bG' \simeq \bG $,
we call $\cU$ its   restriction to $\Sigma $  and then we lift it to $\Gamma$: $\cU_{\Gamma}:=
\pi_2^*(\cU).$
There is   a natural map 
of sheaves on $\Gamma $:
\begin{equation*}
 \Lambda ^2 \mathcal U_{\Gamma } \lra  H^0(F,\omega_F)\otimes \cO_{\Gamma }=
\Lambda^2R^1 \otimes \cO_{\Gamma}. 
\end{equation*}
Let $\cE$ be the cokernel of this map. The adjoint class map  amounts to an  arrow:
\begin{equation*}
 \alpha: \Lambda ^3\cU_{\Gamma } \otimes \pi _1^*  \cO _{\bP (R^3)}(-1) \lra \cE.
\end{equation*}
More precisely, the adjoint class is a global section of the sheaf $Im(\alpha)$ supported on $\Gamma $. We use the bijection $\Gamma \lra \Sigma$ to think of the preceding construction 
as taking place on $\Sigma$. We deal next with the vanishing locus of $\alpha $. \end{rem}

The main result of this section is:
\begin{thm}\label{dddouble_lines}
The adjoint class map  vanishes exactly on  $\Sigma \cap F$, i.e. on the  locus of double lines on $V$.
\end{thm}

\begin{proof} 
 We divide the proof into three parts.
 
\noindent {\it First case:} We assume that $r\in F\cap \Sigma $. We denote  $\Omega_{ij}=\eta_i \wedge \eta_j$. This is the basis of $ W(r)^2$ given above. 
By Corollary (\ref{base_locus_r_in_F}), the three corresponding canonical divisors 
contain the curve $C_r$. Hence:
\begin{equation}
 W(r)^2 :=<\Omega_{12},\Omega_{23},\Omega_{31}>\subset H^0(F,\omega _F(-C_r)).
\end{equation}

Now we use the following result of Tjurin:

\begin{prop} (cf. [T], Corollary (2.2)). \label{pending}
 For any $r\in F$, $h^0(F,\omega _F(-C_r))=3$ and the linear system $\vert\omega_F(-C_r)\vert $ 
has no base divisor.
\end{prop}

Now   $C_r \subset D_r $, and then $\omega_{\eta_1,\eta_2,\eta_3,\xi}\in H^0(\omega_F-(C_r)).$ 
But $H^0(\omega_F(-C_r)) =  W(r)^2$ by the previous
proposition.  Therefore  
we have the requested  vanishing.

\noindent {\it Second case:} We assume that  $r  \in \Sigma -  F$  and we require that $ l_r \cap V$  
contains no  Eckardt points. By   Proposition (\ref{finite_base_locus})
the linear system $\vert W(r)^2 \vert$
has no base divisor and then  from  Theorem (\ref{adjoint_thm}), b) we see
that the   adjoint class  $  [\omega _{\eta_1,\eta_2,\eta_3,\xi}] $ is non trivial in this situation.

\noindent {\it Third case:} Finally we assume that   $r\in \Sigma -F$ and $l_r \cap V $ is a finite set, which  contains some Eckardt points, to be called  $p_i.$ Consider the elliptic curves  $E_i \subset F$, which parametrize the lines through our  points $p_i$.  Clearly   $E_i  \cap E_j = \emptyset $, because otherwise    $l_r \cap V  = l_r.$  We need the following result.

\begin{lem}
In this case  the base divisor $B$ of the  linear system $ W(r)^2 $
 is reduced and  supported on the union of the curves  $E_i$'s. 
\end{lem}
\begin{proof}
Consider a line  $l_r \not\subset  V$ which we assume to intersect  $V$ in an Eckardt point. Our aim here is to prove  the statement that   the   linear system $W(r)^2$
has reduced  base divisor along the Eckardt curve $E \subset F$.
In order to prove reducedness  it is enough to find a point $ s \in E$ where  the tangent plane $T_s(F)$  is not contained in one of the hyperplanes which cut on $F$ our linear system $W(r)^2.$
Those are   the Schubert   hyperplanes 
which give the  condition of non empty intersection with  a plane  containing  the line $l_r$.
If the plane is  $  x _1 = x_3 = 0,$  then 
the Schubert hyperplane has equation $p_{1,3 } = 0.$
Our    $V$  has   equation  $x_4 Q + K (x_1, x_2, x_3 ) = 0 $.
Up to a linear change of the coordinates, we may assume that on the plane $x_4 = 0 = x_0$ the line 
$x_2 = 0 $ is tangent to the cubic $K= 0$ at the point  $[1:0:0]$.
We choose our line $l_s$ to be $ (\ast,\ast,0,0,0 )$, then   $l_s$ is in $V$ and  it contains  the Eckard point $[1:0\dots:0].$  
By setting $p_{0,1} =1 $,    one can pass to  an affine chart in the Grassmannian,  using  the other   Pl\"ucker coordinates $p_{0,j}$  and   $p_{1,j}$, and $s$ is the origin.
In our situation  it  is a matter of routine  to check  that the tangent plane $T_s(F)$ is  given 
in the said  chart as the  coordinate   plane  $ (p_{0,3 }  , p_{1,3} )$, and therefore it is  not contained in the Schubert hyperplane. 
\end{proof}

The next two propositions  complete  the proof of our theorem. Consider 
\begin{equation*}
0 \to T_F  \to T_F(B) \to T_F \otimes \cO_B(B)\to 0 
\end{equation*}
so that  one has the exact sequence:
\begin{equation*}
 \dots \to H^0 (B,T_F\otimes \cO_B(B)) \to H^1(F,T_F)\to H^1(F,T_F(B)) \to \dots
\end{equation*}
\begin{prop} The arrow  $H^1(F,T_F) \to H^1(F,T_F(B))$ is injective, because   
\begin{equation*}
 H^0 (B, T_F \otimes \cO_B(B)) = 0.
\end{equation*}
\end{prop}

We remark  that  $H^0 (B, T_F \otimes \cO_B (  B   )) = \oplus  H^0 (E_i , T_F \otimes \cO_{E_i} (  E_i   ))$.
This is due to $B$ being  the disjoint union of the $E_i$. Then the proof of the proposition descends from:

\begin{prop} For any  Eckardt  curve $E \subset F$
\begin{equation*}
 H^0 (E, T_F \otimes \cO_E (  E   )) = 0 \, .
\end{equation*}
\begin{proof}  
We recall that $T_F$ is the restriction to $F$ of the universal rank $2$ sub-bundle on $\bG$.
Moreover $\cO_E(E) \simeq  \cO_E(-1),$ by the adjunction formula on $F$, since 
$\omega_F \simeq \cO_F (1).$
Our elliptic curve $E$ represents the lines passing through a fixed point $p$ contained in $V$, 
and therefore they are  contained in the tangent space $T_{p} (V) $. 
The lines through $p$ and contained in $T_{p} (V) $ are parameterized by any plane 
$\pi \subset T_{p} (V)$ not containing $p$. Over our plane, seen as a  subvariety of $\bG$, 
there is a bundle $\mathcal B$  which is  the restriction of the universal rank $2$ subbundle.  
This bundle $\mathcal B$ comes with a subbundle $ \mathcal L $ of rank $1$, corresponding to 
the geometric condition of passing through $p$, and further it has a second subbundle, $\mathcal G$, 
which gives the intersection of the line with the plane $\pi \subset T_{p} (V).$  In other words there is
a splitting:
$\mathcal L  \oplus \mathcal G \simeq \mathcal B$.
Remember that the universal subbundle  is in fact a subbundle of the  trivial bundle of rank $5$, 
and therefore  both line bundles come with a non trivial  morphism to the trivial bundle. Since 
$det(\mathcal B) = \cO_{\mathbb P^2} (-1)$, we must have   $\mathcal G \simeq \cO_{\mathbb P^2} $  and 
$\mathcal L=\cO_{\mathbb P^2} (-1)$. In conclusion:  $ \cO(E, T_F \otimes \cO_E(E)) \simeq 
 \cO_{E}(-2) \oplus   \cO_{E}(-1) $,
hence, $H^0 (E, T_F \otimes \cO_E ( E)) = 0 $.
\end{proof}
\end{prop}

The proof of our  Theorem is thus completed.
\end{proof}

%%%%%%%%%%%%%%%%%%%%%%%%%%%%%%%%%%%%%%%%%%%%%%%%%%%%%%%%%%%%%%%%%%%%%%%%%%%%%%%%%%%%%%%%%%%%%%%%%%%%%%%%%%%%%%%%%%%%%%%%%%%%%%%%%%%%%%%%%%%%%%%%%%%%%%%%%%%%%%%%%%%%%%%%

\subsection{  
Explicit computation of the adjoint form for a transverse  line}\label{explicit_adjoint} In this section we compute the adjoint form  when the first order deformation  $\xi $ is associated with a special line
$l_r$ which  intersects $V$ in a least two different points.  We begin with a cohomological computation, in the proof of which we need to appeal to a simple fact: the scheme $W_p \subset F $   which parameterizes  the lines through a   point $p \in V$ is in fact a planar scheme but it is  not supported on any line in $ \mathbb P^9$.
\begin{lem}
 For any line $l_r$ as above one
 has:
\begin{equation*}
 \dim H^0(F,\omega _F \otimes \cI_{D_r})=4.
\end{equation*}
\end{lem}

\begin{proof}
Let $A$ and $B$ be two points in $ V \cap  l_r$.  The   lines through $A$ in $V$ are   contained in     the projective tangent space $T_A(V)$. 
Similarly for $B$ and the remaing point $C$ (if any) on  $V\cap l_r$. 
Since $r \in \Sigma $, then the image of $l_r$ under the dual map is a line (see Remark (\ref{dual})), 
thus we have 
that  $T_A(V)  \cap T_B(V) \cap T_C(V)$ is a $2$-plane $\pi$, the base of the dual pencil.  The lines in $V$  through $A$ 
intersect $\pi$, and the same is true for the lines passing through  $B$ and for  $C$.
Hence the Schubert hyperplane associated with the base $\pi$ provides a canonical divisor $K_{\pi}$ which supports our scheme $D_r.$  This is true for a general line $r$ of second type, and then the same holds for every $r \in \Sigma$, by continuity.  Now  the assumption $l_r  \not \subset  V$ gives  that $r \not \subset  \pi$, and therefore      $K_{\pi}  \not \in \vert W(r)^2 \vert $, which is the system  cut on $F$  
 by the Schubert divisors  associated with the 2-planes containing $l_r$.   
Therefore $h^0(F,\omega _F \otimes \cI_{D_r})\ge 4$, for     $r \not \in F$ .  

On the other hand consider the
 plane $\pi_A\subset \bG$ parameterizing the lines through $A$ and  contained in the tangent  hyperplane $T_A(V)$.
We have recalled  above     that  the scheme $W_p $ parameterizing the  lines through $p$ is a subscheme of  $\pi_A$, but it is not supported on any line of $\pi_A$. A hyperplane of $\bP ^9$ supporting 
$D_r$ must then contain $\pi_A$ and for the same reason contains  $\pi_B$. Now $\pi_A \cap \pi_B =\emptyset $ and therefore  the  linear space they span is of  dimension
$5$. This shows that $h^0(F,\omega _F \otimes \cI_{D_r})\le 4$ hence  the  proof is completed.
\end{proof}

The useful outcome is then the property that  the adjoint class takes value  in a space of  dimension $1\, $ which is
$ H^0(F,\omega_F \otimes \cI_{D_r})/W(r)^2.$

We can write down explicitily a generator of this space in terms of a basis in $R^1$:
let $l_r$ be the line  $z_0=z_1=z_2=0$ and $A=(0:0:0:1:0)$, $B=(0:0:0:0:1)$, i.e.
$W(r)=\langle z_0,z_1,z_2 \rangle \subset R^1=S^1$.
 We can assume that 
$T_A(V),T_B(V)$ are given by $z_4=0$ and $z_3=0$ respectively. Identifying
$R^1$ with the space of $1$ forms on $F$ we put $\omega_{i+1}$ for the form corresponding to 
the hyperplane $z_i=0$. Then 
the canonical divisors attached to the planes $z_0=z_1=0$, $z_0=z_2=0$ and $z_1=z_2=0$ containing
$l_r$  generate $W(r)^2$ and correspond to $\omega_1\wedge \omega _2, \omega_1\wedge \omega _3,\omega_2\wedge \omega _3$.
On the other hand
the lines intersecting
$T_A(V)\cap T_B(V): z_3=z_4=0$ give the $(2,0)$ form $\omega_4\wedge \omega _5$.
We get that:
\begin{prop}\label{what} 
For this deformation $\xi$ the adjunction class  $[\omega _{\omega_1,\omega_2,\omega_3,\xi}] $ 
 is represented in
\begin{equation*}
 H^0(F,\omega_F \otimes \cI_{D_r})/W(r)^2
\end{equation*}
 by $\rho \, \omega_4\wedge \omega_5$, $\rho \ne 0$.   \end{prop}
   
\begin{rem} 
 It is easy  to see  that a special line $r \notin F$   intersects $V$ in either $3$ points 
or else $V\cdot  l_r  = {3P} $. Indeed if $l_r$ of equation    $  z_0 =  z_1 =  z_2 = 0$,
is simply tangent to $V$ at the point $P:=$  $[0:\dots:0:1] $ then 
we may assume that the equation for  $V$
 is  of type 
$E = A z_0 + B z_1 + C z_2 + z_3 ^ 2 z_4 .$ We have that the restriction of the quadrics  $ \partial E / \partial z_3 $ e $ \partial E / \partial z_4 $
are  independent  on  $l_r$. Still, they vanish on $P$,  and since $l_r$ is special, then 
it must be the case that all $\partial E / \partial z _j $ vanish on $P$, which then should be singular on $V$.
One can check that the points $P$ as above, where a special line is $3-$tangent to $V$,  are the points on the intersection with the Hessian hypersurface associated to $V.$  These points have the property that the tangent space cuts $V$ in a surface with singularity either of biplanal or uniplanar type (or else $P$ could be an Eckardt point).
The special line in the biplanar situation  is the line through $P$ which is the intersection of the two  components on the tangent cone, while in the uniplanar case any line of the pencil of base $P$ is special.

It is apparent from the preceding discussion that in general the adjoint class is represented by the canonical divisor $K_{\pi}$ corresponding to the  plane $\pi$, base of the dual pencil. By continuity therefore the adjoint class for $r \in \Sigma$  should always be represented by the same type of Schubert  divisor.
\end{rem}

%%%%%%%%%%%%%%%%%%%%%%%%%%%%%%%%%%%%%%%%%%%%%%%%%%%%%%%%%%%%%%%%%%%%%%%%%%%%%%%%%%%%%%%%%%%%%%%%%%%%%%%%%%%%%%%%%%%%%%%%%%%%%%%%%%%%%%%%%%%%%%%%%%%%%%%%%%%%%%%%%%%%%%%%

\section{Infinitesimal invariants}
First we recall briefly  the definition of the infinitesimal invariant of a normal function associated with a varying family of cycles. This is a theme of  Griffiths, investigated by him, M. Green and C. Voisin \cite{GriffIII}, \cite{Green1}, \cite{Green2} and \cite{Voisin2}.
Next we turn to \cite{PiZu} so to compute this invariant by means of the method of the adjunction form.
Finally we can  apply our preceding  results   to detect the properties  of the normal function associated with  the cycle $F-F^-$ in $JV$, and draw as a consequence 
the statements about the Abel-Jacobi class of this cycle.

%%%%%%%%%%%%%%%%%%%%%%%%%%%%%%%%%%%%%%%%%%%%%%%%%%%%%%%%%%%%%%%%%%%%%%%%%%%%%%%%%%%%%%%%%%%%%%%%%%%%%%%%%%%%%%%%%%%%%%%%%%%%%

\subsection{Definition of the infinitesimal invariant}

We state some facts, mostly for our notational convenience. A complete reference is \cite{Voisin3}.

Let $\rho: \cX\lra B$ be a smooth projective morphism onto a smooth base $B$. Let $\cH^k$ be the sheaf
$R^k \rho_{*} \mathbb C \otimes \cO_B$ with the holomorphic subbundle $F^i \cH^k$ determined by the Hodge filtration on each fiber \label{infinv}
$H^k(X_b,\mathbb C)$. The $(2p-1)-th$ intermediate Jacobian of $X_b$ is defined as: 
\begin{equation*}
 J^{2p-1}(X_b):=H^{2p-1}(X_b,\mathbb C)/(F^pH^{2p-1}(X_b, \mathbb C)+H^{2p-1}(X_b,\mathbb Z)). 
\end{equation*}
All these tori fit in a family over $B$. The sheaf of holomorphic sections of this family is
\begin{equation*}
 \mathcal J^{2p-1}:=\cH^{2p-1}/(F^p\cH^{2p-1}+\cH^{2p-1}_{\mathbb Z}),
\end{equation*}
where $\cH^{2p-1}_{\mathbb Z}$ is the locally constant sheaf on $B$ with fiber $H^{2p-1}(X_b,\mathbb Z)/(\text{torsion})$ over $b$.

A codimension $p$ cycle $\cZ$ on $\cX$  
such that $\cZ_b\equiv_{hom}0,\, \forall b\in B$,
defines  a {\em normal function} $\nu _{\cZ}:B \lra \cJ^{2p-1}$
sending $b$ to the Abel-Jacobi image of $\cZ _b$. The Gauss-Manin connection
\begin{equation*}
 \nabla: \cH^{2p-1} \lra \cH^{2p-1}\otimes _{\cO_B} \Omega^1_B
\end{equation*}
satisfies  the property that $\nabla (F ^i\cH ^{2p-1}) \subset F^{i-1} \cH^{2p-1} \otimes \Omega ^1_B $ (Griffiths' transversality theorem). Hence
it induces 
\begin{equation*}
 \nabla^{p-1}_p :\frac {F^p\cH^{2p-1}}{F^{p+1}\cH^{2p-1}} \lra \left(\frac {F^{p-1}\cH^{2p-1}}{F^p\cH^{2p-1}}\right)
 \otimes _{\cO_B}\Omega^1_B.
\end{equation*}

Let $\tilde \nu_{\cZ}:B\to \cH^{2p-1}$ be a a local lifting of $\nu_{\cZ}$. The infinitesimal invariant  $\delta _{\cZ}$ is
defined as the section of $Coker (\nabla^{p-1}_p)\to B$ given by the composition $\nabla \circ \tilde \nu_{\cZ}$. The non-triviality of $\delta_{\cZ}$ implies that $\nu_{\cZ}$ is not a torsion section. 

%Let $\theta$ be the natural principal polarization on $JV.$
%%%%%%%%%%%%%%%%%%%%%%%%%%%%%%%%%%%%%%%%%%%%%%%%%%%%%%%%%%%%%%%%%%%%%%%%%%%%%%%%%%%%%%%%%%%%%%%%%%%%%%%%%%%%%%%%%%%%%%%%%%%%%%%%%%%%

\subsection{The  Fano surface in its Albanese variety}
We return to our cubic threefold  $V\subset \bP^4$ and to the Fano surface $F$ of the  lines  on $V.$ There is a universal morphism from  the Albanese variety $Alb(F)$ to the  intermediate Jacobian $JV$, due to the universal properties of the Albanese map and the use of the restriction of the Abel-Jacobi map to the $1-$cycles represented by  the lines.  More precisely, the two maps, which we obtain by  fixing a point  $r\in F$, in fact  coincide  \cite{CG}:  \[
a_r: F\to Alb(F) = JV.
\]
Moreover  in this way $F$ is embedded in $J(V)$.   By abuse of language we call $F\subset JV$
 the cycle $a_r(F).$ Composing the Albanese map with the $-1$ involution on $JV$ we get another embedding $a_r^-: F\to JV$ and we  let $F^-$ be the image. 
The $-1$ involution acts trivially on the even cohomology, hence we see that the cycle $F-F^-$
is homologically equivalent to zero. This  is  completely analogous to the Ceresa cycle, see  \cite{Ce},  where it  is  dealt with the curve in its Jacobian. 
Let 
\[ 
 J^5(JV) =\frac{H^{0,5}(JV)\oplus H^{1,4}(JV)\oplus H^{2,3}(JV)}{H^5(JV,\bZ)}
\]
 be the intermediate Jacobian associated to the fifth cohomology of $JV.$ In dealing with the Abel-Jacobi map, $AJ$, it is convenient to use the equivalent definition: 
\[ 
 J^5(JV) =\frac{(H^{5,0}(JV)\oplus H^{4,1}(JV)\oplus H^{3,2}(JV))^*}{H_5(JV,\bZ)}.
\]
The Abel-Jacobi invariant of the Fano cycle is defined by us to be:
\begin{equation}\label{nu_r}
\nu_r(V)=AJ (F-F^-)=\int _{F^-}^F\in J^5(JV).
\end{equation}

% we recall that 
% \[ 
%  H^5(JV,\bZ)\simeq \bigwedge^5 H^3(V,\bZ).
% \] 
% Given   family  $\pi_{\cV}: \cV\to B$ of smooth cubic threefold  the normal function that we use     below is a section $ \nu $  over  $B$  with value  $\nu(b)=\nu_r(V_b)$.
%%%%%%%%%%%%%%%%%%%%%%%%%%%%%%%%%%%%%%%%%%%%%%%%%%%%%%%%%%%%%%%%%%%%%%%%%%%%%%%%%%%%%%
%%%%%%%%%%%%%%%%%%%%%%%%%%%%%%%%%%%%%%%%%%%%%%%%%%%%%%%%%%%%%%%%%%%%%%%%%%%%%%%%%%%%%%
%%%%%%%%%%%%%%%%%%%%%%%%%%%%%%%%%%%%%%%%%%%%%%%%%%%%%%%%%%%%%%%%%%%%%%%%%%%%%%%%%%%%%%

\subsection{The Primitive Intermediate Jacobian}
The  natural principal polarization $\theta$ of $JV$ determines the Lefschetz splitting
\[ 
H^5(JV,\bZ)=\theta^2 H^1(JV,\bZ)\oplus \theta P^3(JV,\bZ)\oplus P^5(JV,\bZ),
\]
where 
\[
 \begin{aligned}
&P^3(JV,\bZ)=\ker(\theta^3:H^3(JV,\bZ)\to H^9(JV,\bZ)),\\  
 & P^5(JV,\bZ)=\ker(\theta:H^5(JV,\bZ)\to H^7(JV,\bZ))
\end{aligned}
\]
are  the primitive cohomology groups.

 On the complex primitive cohomology there is  the Hodge decomposition:
\[
P^5(JV,\bC)=\sum_{i+j=5}P^{i,j} 
\]
 where $P^{i,j}=\ker \theta: H^{i,j}(JV)\to H^{i+1,j+1}(JV)$.
 In particular
 $P^{5,0}=H^{5,0}(JV).$ Let $W\subset H^{1,0}(JV)$ be a $3-$ dimensional space and  $L\subset H^{0,1}(JV)$ be
 the annihilator of $W,$ so $\dim L=2.$   Consider the  inclusion  $\bigwedge^3 W\otimes H^{0,2}\subset H^{3,2}(JV).$
\begin{lem}\label{lem:wedge}
 We have the equality:
\[
\bigwedge^3 W\otimes \bigwedge^2 L=(\bigwedge^3 W\otimes H^{0,2}(JV))\cap  P^{3,2}.
\]
\end{lem}
\begin{proof}
This is a standard computation in linear algebra. We give some details for the reader's  convenience.
A polarization on   $H^{1}(JV, \mathbb C) $ amounts to an alternating form $Q(\alpha, \beta) $ such that the Hermitian form  $H(\alpha,   \beta ) := i Q(\alpha, \bar \beta)$ is positive definite on $H^{1,0}.$
We can fix an orthonormal basis $\{\omega_i\}_{i=1...5}$ of $H^{1,0}(JV)$ such that $W=<\omega_1,\omega_2,\omega_3>.$ By construction we have that the polarization can be represented by an element of the form:
\[
\Theta= -i \sum \omega_i\wedge\overline{\omega}_i.
\]
Now $P^{3,2}=\ker (\theta: H^{3,2}(JV)\to H^{4,3}(JV)),$ and the map 
is given by the cup-product with $\theta.$ An easy  count  yields that
$\ker(\theta)\cap (\bigwedge^3 W\otimes H^{0,2})$ is generated by the decomposable form
\begin{equation} \label{primsiprimno} 
\Omega_W=\omega_1\wedge\omega_2\wedge\omega_3\wedge\overline \omega_4\wedge \overline \omega_5, 
\end{equation} 
where 
$\{\omega_1,\omega_2,\omega_3\}$ is a basis of $W$ and $\{\overline{\omega_4},\overline{\omega_5}\}$ a basis of $L.$ This gives the equality.
\end{proof}

We can define the primitive intermediate Jacobian 
\[
P^5(V)= \frac{H^{0,5}(JV)\oplus P^{1,4}\oplus P^{2,3}}{P^5(JV,\bZ)}
\]
and similarly for $P^3(V)$.
From the Lefschetz decomposition above 
we get then a corresponding  decomposition for our  intermediate Jacobian
\begin{equation} J^5(JV)= \theta ^2 JV \oplus \theta P^3(V) \oplus P^5(V).\label{deco1}\end{equation}
 Hence the definition (\ref{nu_r}) becomes:
  \begin{equation}\label{deco3}
  \nu_r(V)=\nu^{1}_r(V)+\nu ^{3}_r(V)+\nu^5_r(V).\end{equation}
In the Proposition \ref{prop:vanishing_mu_3} below we prove that $\nu^3_r(V)=0$ and that $\nu^5_r(V)=\nu^5(V)$ does not depend on $r$. We call $\nu^5(V)$ the primitive 
Abel-Jacobi invariant of the cycle.
To prove this Proposition we need the following result:

\begin{prop}\label{prop:translations} Consider in  $JV$ an algebraic  cycle  $Z$ of class  $[Z]= c\, \theta ^{p}$, with $p = 3$ (resp. $p = 4$),  and  take a translate  $Z_a$.
Then the projection inside of $P^5(V)$ (resp. inside $ \theta ^2  P^3 (V)$) of the  Abel-Jacobi invariant of $Z_a - Z$ vanishes.
\end{prop}
\begin{proof}
Let  ${\ast }$  be  the Pontrjagin product. We first remark  that
\begin{equation}\label{Pontrjagin}
   [Z] {\ast} H^{9} (JV,\mathbb C) \subset \theta^{p-1}  H^{1} (JV,\mathbb C).
\end{equation}
This is a consequence of  the following facts,   cf. \cite{lieberman}:
(i) the hypothesis yields  $[Z] = k( \theta ^{4})^{\ast \, (5-p)}\,$ (ii)
the action on  $H^{\ast} (JV,\mathbb C)$   given by   ${\ast } $ with $\theta ^{4} $ is a multiple of the operator  $\Lambda$ dual to the Lefschetz operator $L(z)= \theta \wedge z.$  
Then (\ref{Pontrjagin})  holds  because $H^{9} (JV,\mathbb C)= \theta ^ {4}  H^{1} (JV,\mathbb C)$.
Given a  form $\omega$  of type $11- 2p$    we look at
the product $ \mu^{\ast}  [\omega]  \wedge  ([\alpha] \times [Z]) $  where  $\mu: JV \times JV \to JV$ is the sum
and $[\alpha] \in H^{9} (JV,\mathbb C)$. 
  \begin{lemma} The  K\"unneth component  of type $(1,10)$ 
 of  $\mu^{\ast} [\omega]\wedge (1_{JV} \times [Z])  $ vanishes  if $\omega$ is primitive and  $p  \neq 5$.
\end{lemma} 
 \begin{proof} One has
 $ \mu_ {\ast} ( \mu^{\ast} [\omega]   \wedge  ([\alpha] \times [Z] )) = [\omega ] \wedge ( [ \alpha ]  {\ast}[ Z])\,
\overset{(\ref{Pontrjagin})}{=}  ([\omega]  \wedge  \theta^{p-1}) \wedge \dots \, =\, 0 $, since $[\omega], $ being primitive, is by definition in the kernel
 of $\theta^{2p-5}$. 
 \end{proof}
To finish the proof of the proposition we take a path  $\gamma$   from the origin  to $a$ so that 
$Z {\ast} \gamma $  is a chain with boundary our  cycle $Z_a- Z$. We show  that 
$\int _{Z {\ast} \gamma }  \omega = 0$ for any
$\omega$ harmonic form  representing a primitive  cohomology class from $ F^ {6-p} H^ {11-2p}$.
Indeed:
\[
\int _{Z {\ast} \gamma }  \omega=  \int _{\gamma \times Z}\mu^{\ast} (\omega),
\]
which we compute by means of Fubini's theorem integrating first along $Z$. Then we note that the remaining integral along $\gamma$ vanishes because we should integrate the  first part  from  the K\"unneth component of type $(1,10)$ of  $\mu^{\ast} [\omega]\wedge (1_{JV} \times [Z])  $, which  is null  in force of   the preceding lemma.
\end{proof}
Now we can prove:
\begin{prop}
\label{prop:vanishing_mu_3}
 For the Fano normal function one has  $\nu^{3}_r(V)=0$. Moreover $\nu^5_r(V)$ does not depend on the choice of the line $r$. \end{prop}
  \begin{proof}
 It is enough to show  $\theta \cdot \nu^{3}_r(V)=0$. By \cite{Voisin3}, proposition 9.23, this is the same thing as  to check   that the cycle $ \Theta \cdot  F - \Theta \cdot F^-  $ has the property that its Abel-Jacobi invariant has null component in  $\theta ^2 P^3 (V)$.  We use 
\[
  \Theta \cdot  F - \Theta \cdot F^- = ( \Theta \cdot  F -  2 C_t )  + 2 (C_t - C_t^- ) + (2 C_t^-  -  \Theta \cdot F^- ),
\]
 where  
$C_t$ is the curve in $F$ of the  lines that intersect  a given line $l_t\in V$.  We claim that the Abel-Jacobi invariant of the first summand belongs to 
  $ \theta^ 3  \cdot Pic^o ( J)$,  which then implies   the vanishing of its second component.
Indeed   (i)   $ \Theta \cdot  F$ and $  2 C_t  $  are homologous  divisors in  $F$ , (ii)  the Albanese embedding of $F$  in $J$  induces   an isomorphism of the Picard varieties and (iii) the cohomology class of $F$  is  $  \theta^ 3 / 3!$ .  The same thing holds for the last summand. On the other hand   $C_t$ (which is a ``Prym curve'' for $JV$) is known to be invariant by $(-1)^*$ up to   translation by some point $\lambda \in J$, so that  $C_t^-=t_{\lambda }^*(C_t)$. Then  Proposition \ref{prop:translations} above applies to $(C_t - C_t^- )  $. 
  The statement about  $\nu^5_r(V)$ is also   proved using the same proposition, because a change in the base line  determines  a translation on the Albanese image of $F$.
\end{proof}

Our next aim is to  show that on a  generic cubic threefold $V$ the value   $\nu^5(V)$ is not a  torsion point  in the corresponding intermediate Jacobian $P^5(V)$.

%%%%%%%%%%%%%%%%%%%%%%%%%%%%%%%%%%%%%%%%%%%%%%%%%%%%%%%%%%%%%%%%%%%%%%%%%%%%%%%%%%%%%%%%%%%%%%%%%%%%%%%%%%%%%%%%%%%%%%%%%%%%%%%%%%%%%

\subsection{The Fano Normal Function}
We deal now  with a family of smooth  cubic threefolds $\pi_{\cV}: \cV\to B$, the corresponding
 Fano surfaces family $\pi_{\cF}:\cF \to B$, and the family of their Albanese varieties $\pi_{\cA}:\cA\to B$ which coincides with the family of  intermediate Jacobians $\pi_{\cJ}:\cJ\to B$. We then  consider the family of intermediate  Jacobians which are built using the  fifth cohomology of the fibres of $\cJ$ (i.e. they are the intermediate Jacobians  of the intermediate Jacobians):
 \begin{equation}\label{int-int}
\pi_{\cJ}: \cJ^5\to B,
 \end{equation}
where the fiber over  $b\in B$ is $J^5(JV_b).$

 The family of  intermediate  Jacobians of   cubic threefolds  $\pi_{\cV}$ depends on the local system $R^3\pi_{\cV}{_\ast}\bZ=\{H^3(V_b,\bZ) \}_{b\in B},$
 while  $\pi_{\cJ} $ depends on the local system $ R^5\pi_{\cA}{_\ast}\bZ$ with  fibres
 $$\{H^5(J(V_b),\bZ)\}_{b\in B} \simeq \{\bigwedge^5 H^3(V_b,\bZ)\}_{b\in B}.$$
 From (\ref{deco1}) we obtain then  a decomposition of the fibration:
\begin{equation}\cJ^5\simeq \cJ\oplus \cP^3\oplus \cP^5\label{deco2}\end{equation}
where  $\pi_3: \cP^3 \to B$ and
$\pi_5: \cP^5 \to B$ denote the obvious  families of complex tori.
 
  We assume for a while that  we are given a section $s:B\lra \cF$ of $ \pi_{\cF}:\cF \to B.$ Note that locally such section always exists. The section $s$ allows us to define the  family of Albanese maps $a_b:F_b\to Alb(F_b)=JV_b$ and then it  allows us to construct two embeddings of  families $\cF \hookrightarrow \cJ$  and  $\cF^- \hookrightarrow \cJ.$
Now  $\cJ$ plays here  the r\^ole of $\cX$  in (\ref{infinv}), the cycle $\cZ$ being then  $\cF-\cF^-,$ with fibre the  cycle $F_b-F_b^-$ which  is homologically equivalent to zero. In this way the use of the Abel-Jacobi mapping provides  a section $\nu_s$ of $\cJ^5,$ which  is a normal function according to  Griffiths (\cite{GriffIII}). In view of Proposition (\ref{prop:vanishing_mu_3}) we have  from (\ref{deco3}):
\begin{equation} \label{deco4} \nu_s=\nu^{1}_s+\nu^{3}+\nu^5 =\nu^{1}_s+\nu^5
\end{equation}

\begin{definition} We call  $\nu=\nu^5 $ the normal function associated to the Fano cycle. Note  that it 
 does not depend on the choice of the  section $s.$
\end{definition}

%%%%%%%%%%%%%%%%%%%%%%%%%%%%%%%%%%%%%%%%%%%%%%%%%%%%%%%%%%%%%%%%%%%%%%%%%%%%%%%%%%%%%%%%%%%%%%%%%%%%%%%%%%%%%%%%%%%%%%%%%%%%%%%%%%%%%

\subsection{Forms, deformations and the computation of the infinitesimal invariant.}
In \cite{PiZu} the infinitesimal invariant attached to an irregular surface in its Albanese variety is computed. We take from there  that the infinitesimal invariant can be seen as a functional on $\mbox{Ker}(\gamma)$,
where 
\begin{equation}\label{functional}
\gamma:   T_{B,b}\otimes H^2(JV_b,\Omega ^3_{JV_b}) \lra H^3(JV_b,\Omega ^2_{JV_b})
\end{equation}
is induced as usual by the cup-product and the Kodaira-Spencer map.
 We will use presently  the main formula \cite[Corollary 5.2.4]{PiZu} (see Theorem \ref{formula} below). In preparation  we need to present some  new considerations on differential forms.    

First we observe that given a deformation $\zeta \in H^1(T_F)=H^1(T_V)= R^3$
% According to \cite{Carlson etc.} the corresponding infinitesimal variation of Hodge structures is controlled by the map $\xi: %H^{1,0}(F)= R^1\to H^{0,1}(F)= R^4.$
there is a natural primitive form $\Omega_{\zeta}\in P^{3,2}(A)$, defined up to a constant, which is given as follows:
 \begin{equation}
\Omega_{\zeta} =\zeta\cdot(\zeta \cdot \Omega),
\end{equation}
where $\Omega$ is a generator of $H^{5,0}(JV)=P^{5,0}.$
Notice that $\Omega_{\zeta}$ is primitive since $\zeta\cdot \theta=0.$ 

Given  a line $r\in \Sigma$ we have  the corresponding first order  deformation of rank $2$, $\xi\in R^3$, and we write  
$K_1(\xi )=W(r)$ (see \S 3).
\begin{lemma}\label{linalg}
For any $r$ and $ \xi$ as above, the form $\Omega_{\xi}$ is a non trivial decomposable primitive form.
More precisely
\[
\langle\Omega_{\xi}\rangle= \langle\Omega_{W(r)}\rangle=\langle\omega_1 \wedge \omega_2 \wedge \omega_3 \wedge \overline{\omega_4} \wedge \overline{\omega_5}\rangle,
\]
where $\{\omega_1,\dots ,\omega_5\}$ is a basis of $H^{1,0}(JV)$ orhogonal with respect to the polarization and such that $\omega_1,\omega_2,\omega_3$ is a basis of $W(r)$ and $\overline {\omega _4},\overline {\omega _5}$ is a basis of the annihilator $L$ of $W(r)$.
 \end{lemma}
\begin{proof}
We complete a basis $\omega_1,\omega_2, \omega_3$ of $W(r)$ by  choosing   a convenient  basis for $\bar L,$ where $L \subset H^{0,1}(F)$ is the annihilator of $W.$ 
One has:
\[
\Omega_{\xi} =\xi\cdot(\xi \cdot \omega_1\wedge\omega_2\wedge\omega_3\wedge\omega_4\wedge\omega_5)=
2\omega_1\wedge\omega_2\wedge\omega_3\wedge\xi\omega_4\wedge\xi\omega_5.
\]
 Then $\Omega_{\xi} \in(\bigwedge^3 W\otimes H^{0,2}(JV))\cap  P^{3,2 }\, .$
From (\ref{lem:wedge}) we see that  $\langle\Omega_{\xi}\rangle=\langle \Omega_{W(r)}\rangle.$
\end{proof}

Given  $r$ as above   we define $\Phi_{\xi }:=\omega_1\wedge\omega_2 \wedge\omega_3 \, .$  Note that the distinguished  primitive decomposable form that appear in Lemma 
(\ref{linalg}) can be written
\[
\Omega_{\xi}= \omega_1\wedge\omega_2 \wedge\omega_3\wedge\overline{\omega_4}\wedge\overline{\omega_5}=\Phi_{\xi }\wedge\overline{\omega_4}\wedge\overline{\omega_5}.
\]

By recalling  that  $\xi\cdot \Phi_{\xi }=0$ we have:
\begin{lemma}
Let $\gamma:H^1(T_F)\otimes H^{3,2}(JV)\to H^{2,3}(JV)$ be the map defined in (\ref{functional}),
then  for all $\beta\in H^{0,2}(JV)$
\[
 \xi\otimes\Phi_{\xi }\wedge \beta \in \ker(\gamma).
\]
\end{lemma}

Now we  evaluate the infinitesimal invariant functional on these tensors, i.e.  we
 compute
\[
\delta \nu  (\xi\otimes\Phi_{\xi }\wedge \beta).
\]
As it was explained above, our normal function is associated with the codimension $3$ cycle  $\cZ=\cF-\cF^-$ which lives   in  the Albanese fibration  $\cJ$ over some suitable   base $B. $ Fixing a point  $b\in B$  we write   $F:=\cF_b,\, JV:=\cJ_b$. Let then  $\xi \in H^1(F, T_F)$ be the class of the infinitesimal deformation
\begin{equation*}
 F_{\varepsilon}\lra Spec \, \mathbb C[\varepsilon]/(\varepsilon^2)
\end{equation*}
 corresponding to a local parameter $t$ on $B$ around $b$.

 In the present situation the main result of \cite{PiZu}  gives:
\begin{thm} \label {formula}
With the preceding    notations:
\begin{equation*}
\delta_{\nu}(b)(\xi\otimes \Phi_{\xi }\wedge \beta)
= 2 \int _F \omega_{\omega_1,\omega_2,\omega_3,\xi}\wedge a_b^\ast(\beta),
\end{equation*}
where $\omega_{\omega_1,\omega_2,\omega_3,\xi}$ is the standard representative (see (\ref{stand})) of the adjoint form defined above in section $\,4$ and $a_b$ is the Albanese map.
\end{thm}

\begin{proof}
\cite{PiZu}(Corollary 5.2.4).  \end{proof}

We are interested in the following straightforward consequence of the theorems \ref{adjoint_thm} and \ref{formula}.
We use the identifications between the Hodge components and the graded pieces of the Jacobian ring (section \S3):

\begin{prop}\label{corR} Let $r \in \Sigma $ and let $\xi \in R^3$ such that
 $K_1(\xi)=W(r)=\langle \omega_1,\omega_2,\omega_3 \rangle \subset R^1$.
 \begin{enumerate}
\item [a)] If $r\in F$ then for all $ \beta \in H^{2,0}(JV)$, $\delta\nu (\xi\otimes(\omega_1\wedge\omega_2\wedge\omega_3)\wedge \beta)=0.$
\item [b)] If $r\in \Sigma \setminus F$ then there exists $ \beta \in H^{2,0}(JV)$ such that $\delta\nu (\xi\otimes(\omega_1\wedge\omega_2\wedge\omega_3)\wedge \beta)\neq 0.$
\item [c)] If $r\in \Sigma\setminus F$ intersects in $3$ distinct points then $\delta\nu^5 (\xi \otimes\Omega_\xi)\neq 0$. Therefore $\delta\nu^5\neq 0.$
\end{enumerate}
\end{prop}
\begin{proof}
As  usual   we identify $H^{i,j}(JV)$ with $H^{i,j}(F)$ when $0\leq i+j\leq 1$. Then we can rewrite the formula of the infinitesimal invariant in the form
\[
\delta{\nu}(\xi\otimes  (\omega_1\wedge\omega_2 \wedge\omega_3)\wedge   \beta)=2\int_F \omega_{\omega_1,\omega_2,\omega_3,\xi}\wedge \beta.
\]
From Theorem \ref{dddouble_lines} we know that the adjoint class vanishes exactly on the locus of the double lines.
\begin{enumerate}
\item [a)] If $r\in F,$  $\delta\nu (\xi\otimes(\omega_1\wedge\omega_2 \wedge\omega_3)\wedge \beta)=0.$
\item [b)] If $r\notin F$ then take $\beta= \overline{\omega_{\omega_1,\omega_2,\omega_3,\xi}}$ we get
$\delta(\xi\otimes (\omega_1\wedge\omega_2 \wedge\omega_3)\wedge   \beta)=-\int \beta\wedge\overline \beta\neq 0.$ 
\item [c)] If $r$ intersects transversally $V,$ by (\ref{what}) we have that 
$[\omega_{\omega_1,\omega_2,\omega_3,\xi}]$ is equivalent to $\eta_4\wedge \eta_5$ where 
$\{\omega_1,\omega_2,\omega_3,\eta_4,\eta_5\}$ is a basis of $H^{1,0}(F).$ Then we have
$$\eta_4\wedge\eta_5 =\sum_{ij}a_{i,j}\omega_i\wedge \omega_j$$ and we get that $a_{4,5}\neq 0.$
Now $\delta\nu^5( \xi\otimes\Omega_\xi) =\delta\nu(\xi\otimes\Omega_\xi)$ since the form is primitive. We compute 
\[ 
\delta\nu^5( \xi\otimes\Omega_\xi)= 2\int_F\eta_4 \wedge\eta_5\wedge \overline {\omega_4}\wedge\overline{ \omega_5}=2a_{4,5}\int_F\omega_4 \wedge\omega_5\wedge \overline {\omega_4}\wedge \overline{\omega_5}\neq 0.
\]
 \end{enumerate}
\end{proof}
\begin{rem} Consider the locus of special deformation
$ \Xi _2\subset \bP R^3$ (see (\ref{def_Xi}) for the definition) and let
$\pi: \cO(-1)\to  \Xi _2$ be the tautological line bundle  Set $U= \cO(-1)\setminus Z$ where $Z$ is the zero section,
$U$ is a fourfold and the map $\pi:U\to  \Xi _2$ is a $\bC^\ast$ bundle.
By fixing  a generator $\Omega$ of $H^{5,0}(A),$ we obtain a  holomorphic map
$f:U\to \bC$ defined by
$$f(\xi)=\delta \nu (\xi\otimes \Omega_\xi)=\delta \nu^5 (\xi\otimes \Omega_\xi).$$  We know from    \ref{corR}  a)
 that $f$ vanishes on $\pi^{-1}(\cD)$ where $\cD \subset F$ 
is  the curve of double  lines  of the Fano surface.  Now $\pi^{-1}(\cD)$ is a surface, but   $f$ must vanish on a divisor. The corresponding lines  are precisely the
non-transverse  ones by  \ref{corR} c ).   
\end{rem}

We can draw  the consequences  of our computation.
\begin{thm} \label{collez}
The following hold.
\begin{enumerate}
\item [a)] The Fano normal function $\nu$ determines $\cD$,  the curve  of   double lines  on  the cubic threefold.
\item [b)] The Fano normal function $\nu$ is not a torsion section.
\end{enumerate}
\end{thm}

%%%%%%%%%%%%%%%%%%%%%%%%%%%%%%%%%%%%%%%%%%%%%%%%%%%%%%%%%%%%%%%%%%%%%%%%%%%%%%%%%%%

\subsection {Algebraic equivalence}
In the next section we will see that \ref{collez} i) yields  a Torelli-like  theorem.
Here we give our first application of the work done:

\begin{thm} \label{thm:alg_equiv}
The following hold.
\begin{enumerate}
 \item [a)] For a generic $V$  there is no proper map $\mu: W \to JV $ from a smooth variety $W$ of dimension $\leq 4$, for which  there is a cycle $Z$ homologous to $0$ in $W$    
with $\mu _{\ast} Z$ Abel-Jacobi equivalent to $F-F^-$.
 \item [b)] (Van der Geer-Kouvidakis,\cite{vdGK}) For such a $V$ the  
cycle $F-F^-$ is not algebraically equivalent to zero on $JV.$
\end{enumerate} 

\end{thm} 
\begin{proof} 
We work locally, so we will consider a Kuranishi family of cubic threefolds $\rho: \mathcal V \to B$. Let   $R^5 \rho_ *\mathbb C\otimes \cO_B= \mathcal H^5  \to B$ be the variation of polarized Hodge structures, with fibre $H^5(JV_b,\mathbb C)$ at $b$. Assume by contradiction that there exists $W$ as described in (a) for a generic $JV_b$. Via an argument using relative Hilbert schemes, cf.\cite[Remark 11.16]{Voisin2}, after shrinking $B$ if necessary, we can find   
a family $\rho':\mathcal W \to B$ with a compatible  map over $B$:  $\tilde \mu:\mathcal W \to \mathcal JV$. There is also a cycle $\mathcal Z$ on $\mathcal W$, flat over $B$ such that $AJ(\tilde \mu_{b *}(\mathcal Z_b))=AJ(F_b-F_b^-)$. The map $\tilde \mu$ defines for any $b$ a Hodge substructure and we can assume that this provides an inclusion of local systems $\mathcal L \subset \mathcal H^5$, where $\mathcal L=\tilde \mu_* R^5\rho'_{*}\mathbb C\otimes \cO _B$. Observe that $\mathcal L_b$ is orthogonal to $H^{5,0}(JV_b)$ in $H^5(JV_b,\mathbb C)$,$\,\forall b\in B$. We fix a generator $\Omega $ in $ H^{5,0}(JV_b)$. 
Since $\mathcal L$ in $\mathcal H^5$ is preserved by the Gauss-Manin connection, 
$\xi\otimes \xi \cdot \Omega=\Omega_\xi$ must be orthogonal to $\mathcal L_b.$ 
Because of our hypothesis,  we remark that  every computation one may need to deal with  in  regard to  $\nu$ and $\delta{\nu} $ can in fact be  performed   assuming that the lift of  $\nu$ takes value in  $\mathcal L$, but then
$\delta\nu(\xi\otimes\Omega_\xi)=0.$
This is the required contradiction with \ref{corR} i).  So we have proved (a), and then (b) is an immediate  corollary. 
\end{proof}

\begin{rem}
 Note that because the Hodge structure on $H^5(J(V),\mathbb Q)_{prim}$ is simple for very general
$V$ by monodromy arguments, the nonvanishing of the primitive Abel-Jacobi invariant also implies that the cycle $F-F^-$ is not homologous to $0$ in a divisor of $J(V)$, since otherwise its Abel-Jacobi invariant would
belong to a proper subtorus of $J^5(J(V))_{prim}$ corresponding to a sub-Hodge structure of level $3$.
\end{rem}

\begin{rem} \label{hain}
Turning to the case  of a curve $C$ consider the surface $C^{(2)}$ which is its  second  symmetric product, and specifically look at  its image,   the surface $W_2 \subset J(C)$.  
The same procedure as above constructs  a normal fuction $\gamma $ which splits as $\gamma ^3+\gamma ^5$ over the moduli space of Jacobian varieties;  it is associated with the  Abel-Jacobi class of the cycle  $W_2- W_2^-$. The results of Hain, Thm 7.13 in \cite{Hain} imply  $\gamma^5=0$ (but $\gamma^3$ is  not of torsion in general).  In genus $5$  one can prove   directly  that the preceding statement (a) fails for   $W_2 -W_2^{-}$. Indeed we show below that  up to a translation  the cycle is homologically equivalent to zero on a suitable theta divisor of $J(C)$. This is a contrast with our situation where $\nu^3$ is $0$ (see Lemma \ref{prop:vanishing_mu_3}) but $\nu^5 $ is not torsion. 
We find very interesting this  difference between the  two families of Abelian varieties of dimension $5$. Moreover the above result suggests that the theory of the representation for the {\em Torelli} group, of the orbifold fundamental group of the moduli space of the cubic threefolds should be very rich.
\end{rem} 
We restrict our attention to a a general curve $C$ of genus $5$,
and consider $C^{(4)}$, the standard symmetric product.
There is a rational involutive correspondence   $j: C^{(4)} \dasharrow C^{(4)}$
which maps an element $D\in C^{(4)}$  to  its Riemann-Roch adjoint $K_C -D$,
and therefore an involution $ \rho:  H^{*} (C^{(4)},\mathbb C )\to H^{*} (C^{(4)},\mathbb C ).$
By choosing a base point $p \in C$  we embed  $C^{(2)}$  in  $C^{(4)}$ and 
define $G(2):=  \overline { jC^{(2)}} $. With the notations of  \cite{ACGH} 
we consider the map $u: C^{(4)} \to J(C)$ and recall the result written there at p. 326:

   \begin{prop}  Let $C^1_4=\{D\in C^{(4)}\, \vert \, \dim \,\vert D\vert \ge 1 \}$. Then $-\theta x+ \theta^2 / 2  \, =[C^1_4 ] $ in $H^4(C^{(4)},\mathbb C).$
 \end{prop}

Here  $x$ is the class of $C^{(d-1)}$  in $C^{(d)}$, while $\theta$ is the restriction of the class of the $\Theta$ divisor on $JC$,
which turns out to be a translation of $W_4$, where we use the notation $W_d:= u(C^{(d)})$.
Now $u( G(2))$  is a translation of $W_2^ -.$    We have:
   \begin{prop}   The   class   
    $[W_2 - u( G(2))]$  
vanishes in  $H_4 (W_4,\mathbb C) .$
 \end{prop}
\begin{proof}
The map  $ u : C^{(4)} \to  W_4$  contracts the surface $C^1_4$ to a curve,
so that   $0 = u_{*} [C^1_4 ] $ in $ H_{4} (W_4).$
On the other  hand \cite{BB}  contains   the formula for the Riemann Roch involution $\rho$.
This was a result contributed by the first author. 
What it says in our case is that $[C^{(2)}]  -  [G(2)]  =      \theta x-  \theta^2 / 2 =  - [C^1_4 ]$.
\end{proof}

 \begin{prop}   The primitive Abel-Jacobi image   of 
    $W_2 - W_2^ {-}$   vanishes in $P^5 (J(C))$.
    \end{prop}
    \begin{proof}
    This is  because  $W_2^ {-}$ and $G(2)$  are translations of each other,
   so that the primitive Abel-Jacobi image of their difference vanishes.
    Moreover  $W_2 - u( G(2))$  comes from a cycle   homologous to zero on $C^{(4)}$, and therefore 
    we also have for it the vanishing of the primitive Abel-Jacobi  image. 
    \end{proof}
\begin{rem}  An extension of the same   argument yields  a proof for  the fact 
that $C^{(2)}$ and   $G(2)$ are cohomologous  in $W_{g-1}$ for any $ g \geq 5$.
\end{rem}

\begin{rem}
Let $\mathcal M_g$ be the moduli space of curves of genus $g$ and $\mathcal H_g\subset \mathcal M_g$ be  the hyperelliptic locus.
The Ceresa normal
function vanishes on  $\mathcal H_g$, and according to a conjecture of Clemens it should be non zero everywhere on
 $ \mathcal M_g \setminus  \mathcal H_g.$
It is a difficult problem  to determine where the Fano normal function vanishes.
Here we give an answer to a simpler, but important, question of G. Pearlstein.
He asked, in analogy with the genus $3$ case, if there is a  divisor $\cD$ in the moduli space $\mathcal M$ of cubic threefolds where the Fano normal function vanishes. Our results implies that this divisor does not exist.
Indeed take $[V]\in \mathcal D\subset \mathcal M$ the general point of any divisor, and identify the tangent space $T'$ to $[V]$ at $\mathcal D$ with an hyperplane of $H^1(T_V).$  It follows that the projective space $\bP(T')$ intersects the threefold $\Sigma$ of the special deformations,  in a surface. Then we can find in $T'$ a special deformation that does not correspond to a double line of $F$. It follows that the infinitesimal invariant is not zero along $\cD$, therefore the normal function cannot vanish. On the other hand  let $\overline{\mathcal M}$ be the closure of $\mathcal M$ in $\mathcal A_5,$ the moduli space of principally polarized Abelian varieties. From \cite{collinoLNM} (see also \cite{casalaina}) we see that $\overline{\mathcal M}$ contains  the locus $H_5$ of hyperelliptic Jacobians. Note that  $H_5$ is a divisor of $\overline{\mathcal M}.$ The flat limits of the Fano surfaces in the intermediate jacobians when they approach an hyperelliptic Jacobian is the Abel-Jacobi image  $W_2$ of the $2-$symmetric product of
the hyperelliptic curve. This follows, for instance, from a theorem of Olivier Debarre \cite{Debarre}.
 In fact the homology class of the limit is a minimal class and Debarre prove that in a Jacobian the minimal classes are given only by the Abel-Jacobi image of symmetric products. It follows that the limits of the primitive normal function
is zero on $H_5.$ 
\end{rem}

%%%%%%%%%%%%%%%%%%%%%%%%%%%%%%%%%%%%%%%%%%%%%%%%%%%%%%%%%%%%%%%%%%%%%%%%%%%%%%%%%%%%%%%%%%%%%%%%%%%%%%%%%%%%%%%%%%%%%%%%%%%%%%%%%%%%%%%%%%%%%%%%%%%%%%%%%%%%%%%%%%%%%%%%%%%
%%%%%%%%%%%%%%%%%%%%%%%%%%%%%%%%%%%%%%%%%%
%%%%%%%%%%%%%%%%%%%%%%%%%%%%%%%%%%%%%%%%%%%%%%%%%%%%%%%%%%%%%%%%%%%%%%%%%%%%%%%%%%%%

\section{Recovering the cubic threefold from the infinitesimal invariant}

The goal of this section is to show how one can  reconstruct the threefold $V$ from the infinitesimal 
invariant studied above. We consider the  first   order deformations associated with the elements of $ \Sigma.$  Working as we did for  (\ref {corR}) we find that  the infinitesimal invariant provides a functional over $\Sigma$ which vanishes exactly where the adjoint class is trivial. In  (\ref{adjoint_as_section}) we noticed that the adjoint class can be considered as a section of a sheaf on $\Sigma \subset \mathbb G $ and we
proved in Theorem (\ref{dddouble_lines}) that this section vanishes exactly on $\cD=\Sigma \cap F$, the curve of double lines on $V$.
So the title of this section amounts to the statement  that  one can reconstructs  $V$ from  $ \cD $, or equivalently  the question  is whether the ruled surface $S_{\cD}$  covered by double lines is contained in a unique cubic threefold.  It is enough to verify that $S_{\cD}$   has degree larger than $9$.    Fano has studied this ruled surface in great detail, see \cite {Fano1}, and he determined all the relevant numerical data, for instance he proved that $S_{\cD}$ has   degree  $90$ in general. We need on the other hand to  check that the degree of  $S_{\cD}$ is always big enough to make it contained in only one cubic threefold.  To  this  aim  we prove first that every component of  $\Sigma \cap F$ is 
reduced (see (\ref{reduced}) below) and then  we find in (\ref{degree90}) that the surface covered by double lines has degree at least $15$. We expect, but do not need,  that indeed  it is always of degree $90$.

Consider the variety $K(V)$ which parametrizes the  set of  flags $(L, \pi)$, where $L$ is a double line on   $V$ and $\pi \cdot V = 2L + R\, .$  Next  lemma yields by a standard argument (we have in mind here the use of local coordinates) 
the   consequence  that $K(V)$  is a non singular curve at the points where $R\neq L$.  Moreover, since $K(V)$  is a fibre in a universal fibration, all such curves are in the same homology class.
By projection   we have the map   $K(V) \to \cD(V)$, which is in fact injective, because a  plane can cut  $V$ with multiplicity $2$ at most along   one line. This says that all curves  $\cD(V)$ are reduced, with the same homology class in $\mathbb G\, .$

We choose  coordinates such that $z_0=z_1=z_2=0$ is 
a double line of $V$ with residual line $z_0=z_1=z_3=0$. 
So the equation   of our fixed   chosen cubic threefold $V_0$ is of the form
\begin{equation*}
E_0 =  z_0 Q_0+z_1 Q_1+z_2^2z_3,
\end{equation*}
where $Q_0,Q_1 \in S^2$, the space of  homogeneous   polynomials of degree $2$  in $z_0,\dots,z_4$.

Consider   the map of affine spaces:
\begin{equation*}
\begin{aligned}
 \Phi: (S^1)^4\times (S^2)^2 & \lra S^3 \\
(L_0,L_1,L_2,L_3,A_0,A_1)&\mapsto L_0 A_0 + L_1 A_1 +L_2^2 L_3. 
\end{aligned}
\end{equation*} 
\begin{lem}Assume that $V_0$ is smooth. Then
 the differential   $d\Phi $ is surjective at $p_0:=(z_0,z_1,z_2,z_3,Q_0,Q_1).$
\end{lem}
\begin{proof}
We compute $d\Phi $ at $p_0$ in the obvious way.

Let $u=(L_0,L_1,L_2,L_3,A_0,A_1)\in T_{p_0}((S^1)^4\times (S^2)^2)$, then
\begin{equation*}
 \frac {d\Phi (p_0+\varepsilon u)}{d \varepsilon}(0)=z_0A_0+L_0Q_0+z_1A_1+L_1Q_1+z_2^2L_3+2z_2z_3L_2 \in S^3 = T_{E_0}(S^3).
\end{equation*}
We need to show that any element in $S^3$ is of the form
$$z_0A_0+L_0Q_0+z_1A_1+L_1Q_1+z_2^2L_3+2z_2z_3L_2$$ 
when the $L_i$'s move  in $S^1$ and $A_i$ in $S^2$. Due to the summands $z_0A_0+z_1A_1$ we can reduce modulo $z_0$ and $z_1$.
Denote by  $C_i=Q_i(0,0,z_2,z_3,z_4),i=0,1$ the corresponding conics. We are reduced to prove that the combinations
$$
M_0C_0+M_1C_1+M_3z_2^2+M_4z_2z_3
$$
 run over all $\mathbb C[z_2,z_3,z_4]_3$ when the $M_i$'s are linear forms in the variables $z_2,z_3,z_4$.

We claim that:
\begin{enumerate}
 \item [a)] There is no common solution of the three quadratic equations $C_0,C_1,z_2^2$
 \item [b)] The conic $z_2z_3$ is linearly independent of $C_0,C_1,z^2_2$.
\end{enumerate}
The claim is a consequence of the smoothness of the given cubic threefold. Indeed, by computing the $5$ partial derivatives of $E_0$
and reducing modulo $z_0,z_1$ we get the system:
$$
C_0=C_1=z_2^2=z_2z_3=0.
$$ 
A solution of $C_0=C_1=z_2^2=0$ would give a singularity of the threefold over the line $z_0=z_1=z_2=0.$
On the other hand a linear combination 
$$
z_2z_3=\alpha z_2^2+\beta C_0+\gamma C_1
$$
would imply the existence of a denerate conic of the form $z_2\cdot (a_3z_3+a_4z_4+a_5z_5)$ in the linear pencil
generated by $C_0,C_1$. Again this produces a singularity over the line $z_0=z_1=z_2=0.$

Part a) of the claim allows to use Macaulay's Theorem for the ideal sheaf $I=\langle C_0,C_1,z_2^2\rangle \subset \mathbb C[z_2,z_3,z_4]$. Put $R_I=\mathbb C[z_2,z_3,z_4]/I$.
Then $R_I^3$ has dimension $1$ and there are perfect pairings:
$$
R_I^i\otimes R_I^{3-i}\lra R_I^3=\mathbb C 
$$ 
(cf. for instance, section (6.2.2) in \cite{Voisin3}). Part b) of the claim says that $z_2z_3$ is not zero in $R_I^2$, thus $z_2z_3 \cdot R_I^1=R_I^3$
and we are done. 
\end{proof}

\begin{cor}\label{reduced}
 Every component of $\cD(V) $ is reduced if  $V $  is non singular. \end{cor}

Consider now the incidence variety
\begin{equation*}
 T=\{(l,x)\in \bG \times \bP^4 \,\vert \, x\in l\},
\end{equation*}
with its  projections $\pi_1,\pi_2$. Given  a curve   $X\subset \bG$
we denote by $S_X\subset \bP ^4$ the surface covered by the lines parameterized by $X$.
In other words $S_X=\pi_2(\pi_1^{-1}(X))$. Standard techniques from  intersection theory, \cite{Fulton}, 
give:

\begin{lem}
 The degree of the cycle $\pi_{2 \ast}(\pi_1^{\ast}(X))$ is equal to  $d(X)$,  the degree of $X \subset  \bG \subset \bP^9$. The degree of  $S_X$ is then found by dividing $d(X)$  by the degree of the map $
 \pi_1^{-1}(X) \to  S_X$.
\end{lem}

\begin{lem}\label{degree90}
The ruled  surface $S_{\cD}\subset \bP^4$   of double lines  is always  of  degree at least $15$.  
\end{lem}
\begin{proof} Since the degree of  $
 \pi_1^{-1}(S_{\cD} )\to  V$ is at most   $6$,   it is then enough to compute that the  degree of $\cD$ is $90$. By recalling that  
the tautological embedding $F\subset \bG\subset \bP^9$
is the canonical map of $F$, we get $deg(\cD)=\cD\cdot K_F$. On the other hand $\cD \in \vert 2 \,K_F \,\vert $ (cf. (10.20) in \cite{CG}), and then 
$deg(\cD)=2K_F \cdot K_F = 2\cdot K_F^2=2 \cdot 45=90$.  
\end{proof}

\begin{thm}\label{torelli}
 There is only one cubic threefold containing $S_{\cD}$. In particular the infinitesimal invariant determines the cubic threefold.
\end{thm}


\begin{thebibliography}{}
\bibitem[ACT]{Allcock}
{\sc D.~Allcock, J.~Carlson, D. Toledo}, {\em The moduli space of cubic threefolds as a ball quotient}, 
 Memoirs of the AMS \textbf{209} (2011). 

\bibitem[ACGH]{ACGH}
{\sc E.~Arbarello, M.~Cornalba, P. A.~Griffiths and J.~Harris}, 
{\em Geometry of Algebraic Curves. Volume I}, 
Grundlehren der mathematischen Wissenschaften,
No. 267. Springer-Verlag, New York-Heidelberg, 1984.

\bibitem[BB]{BB} 
{\sc P.~Bressler  and J.L.~Brylinski},
{\em On the singularities of theta divisors on Jacobians},
J. Algebraic Geom. {\textbf 7} (1998), pp.~781-796.

\bibitem[CGGH]{Carlson etc.}
{\sc J.~Carlson, M.~Green, P.~Griffiths, J.~Harris}, 
{\em Infinitesimal variation of Hodge structures I}, Compositio Math. \textbf{50} (1983) pp. ~109--205.

\bibitem[CL]{casalaina}
{\sc S.~Casalaina-Martin, R.~Laza}, 
{\em The moduli space of cubic threefolds via degenerations of the intermediate Jacobian},
J. reine angew. Math. \textbf{633} (2009), pp.~29--65.

\bibitem[Ce]{Ce}
{\sc G.~Ceresa}, {\em $C$ is not algebraic equivalent to $C^{-}$ in its
Jacobian}, Ann. of Math. \textbf{117} (1983), pp. ~285--291.

\bibitem[CG]{CG}
{\sc C.H.~Clemens, P.A.~Griffiths},
{\em The intermediate Jacobian of the cubic threefold}, Ann. of
Math.  \textbf{95} (1972), pp.~281--356. 

\bibitem[Co]{collinoLNM}
{\sc A.~Collino}, 
{\em The fundamental group of the {F}ano surface. {I}, {II}},
  Algebraic threefolds (Varenna, 1981), Lecture Notes in Math., vol. 947,
  Springer, Berlin, 1982, pp.~209--218, 219--220.

\bibitem[CP]{CP}
{\sc  A.~Collino and G.~Pirola},
{\em  Griffiths'  infinitesimal invariant  for a curve in its jacobian}, Duke Math. J. \textbf{78} (1995), pp.~59--88.  

\bibitem[De]{Debarre}
{\sc O.~Debarre} 
{\em Minimal Cohomology Classes and Jacobians}, J. Algebraic Geom. \textbf{4} (1995), pp.~321--335.

\bibitem[Do]{Dol}
{\sc I.~Dolgachev,} {\em Topics in Classical Algebraic Geometry, Part I}, \
Lecture Notes, 2010, available online at \
{\tt www.math.lsa.umich.edu/$\sim$idolga/topics.pdf}.

\bibitem[Fa]{Fakhr}
{\sc N.~Fakhruddin},
{\em Algebraic cycles on generic Abelian varieties}, Compos. Math. \textbf{100}
(1996), pp.~101-119.

\bibitem[F]{Fano1}
{\sc G.~Fano} {\em Ricerche sulla variet\`a cubica generale dello spazio a quattro dimensioni e sopra i suoi spazi pluritangenti}, Ann. Mat. Pur. Appl. \textbf{10} (1904), pp. ~251-285.

\bibitem[Fu]{Fulton}
{\sc W.~Fulton}, {\em Intersection Theory},
Ergebnisse, 3. Folge, Band 2, Springer Verlag (1984).


\bibitem[GK]{vdGK}
{\sc G.~van der Geer, A.~Kouvidakis},
{\em The rank-one limit of the Fourier-Mukai transform},
Doc. Math. \textbf{15} (2010), 747--763. 

\bibitem[Gr1]{Green1}
{\sc M. ~Green}, {\em Griffiths' infinitesimal invariant and the Abel-Jacobi map}, J. Differ. Geom. \textbf{29} (1989), pp.~545--555.


\bibitem[Gr]{Green2}
{\sc M. ~Green}, {\em Infinitesimal methods in Hodge theory. Algebraic cycles and Hodge
theory (Torino, 1993)},  L.N.M., 1594, Springer, Berlin (1994), pp.~1--92.

\bibitem[GH]{Principles}
{\sc P.A. ~Griffiths, J. ~Harris},
{\em Principles of Algebraic Geometry}, New York, Wiley and Sons, 1978.

\bibitem[G1]{Bombay}
{\sc   P.A. ~Griffiths},
{\em Some results on algebraic cycles on algebraic manifolds},
Algebr. Geom., Bombay Colloq. 1968,(1969), pp.~ 93-191.

\bibitem[G2]{PRI}
{\sc P.A.~Griffiths},
{\em On the periods of rational integrals, I and II }, Ann. of Math. \textbf{90} (1969), pp.~ 460-541.

\bibitem[G3]{GriffIII}
{\sc P.A. ~Griffiths},
{\em Infinitesimal variations of Hodge structures III: determinantal varieties
and the infinitesimal invariant of normal functions}, Compos. Math. \textbf{50} (1983), pp.~267--324.

\bibitem[HL]{Hain}
{\sc R. ~Hain, E. ~Looijenga},
{\em Mapping class groups and moduli spaces of curves},
Algebraic Geometry, Santa Cruz, 1995: Proc. Symp. Pure Math 62.2 (1997), pp.~97-142

\bibitem[L]{lieberman} Lieberman, D. I. : {\em Numerical and homological
equivalence of algebraic cycles on Hodge manifolds.\/}
Amer. J. Math. {\bf 90} (1968), 366-374.

\bibitem[N]{Nori}
{\sc M.V.~Nori},
{\em Algebraic Cycles and Hodge Theoretic Connectivity}, Invent. Math. \textbf{111}
(1993), pp.~349--373.

\bibitem[PZ]{PiZu}
{\sc G.P.~Pirola, F.~Zucconi},
{\em Variations of the Albanese morphisms}, J. Algebraic Geom. \textbf{12} (2003), pp.~535--572. 

\bibitem[T]{T}
{\sc A.N.~Tjurin},
{\em The Geometry of the Fano  surface of a non singular cubic threefold $F\subset \bP ^4$ and Torelli Theorems for Fano
surfaces and cubics}, Izv. Akad. Nauk SSSR Ser. Mat. \textbf{35} (1971), pp.~498--529. 

\bibitem[V1]{Voisin1}
{\sc C.~Voisin}, {\em Th\'eor\`eme de Torelli pour les cubiques de $\mathbb P^5$}, Invent. Math. \textbf{86} (1986), pp.~577--601.

\bibitem[V2]{Voisin2} {\sc C.~Voisin}, \emph{Une remarque sur l'invariant infinitesimal des functions normal}, C. R.Acad. Sci. Paris Ser. I \textbf{307} (1988), pp. 157-160. 

\bibitem[V3]{Voisin3}
{\sc C.~Voisin}, {\em Hodge theory and complex algebraic geometry. I-II}, 
Cambridge Studies in Advanced Mathematics, 7. Cambridge University Press, Cambridge, 2007.

\end{thebibliography}
\end{document}